\documentclass[11pt]{article}
\usepackage[utf8]{inputenc}
\usepackage[T1]{fontenc}
\usepackage{comment}
\usepackage{authblk}
\usepackage{centernot}
\usepackage{relsize}
\usepackage{changepage}
\usepackage{mathtools}
\usepackage[english]{babel}
\usepackage{amsmath}
\usepackage{amsthm}
\usepackage{amsfonts}
\usepackage{amssymb}
\usepackage{graphicx}
\usepackage[usenames,dvipsnames]{xcolor}
\theoremstyle{plain}
\newtheorem{theorem}{Theorem}[section]
\newtheorem{corollary}[theorem]{Corollary}

\newtheorem{proposition}[theorem]{Proposition}
\newtheorem{lemma}[theorem]{Lemma}

\theoremstyle{definition}

\newtheorem{remark}[theorem]{Remark}

\usepackage{empheq}

\makeatother
\definecolor{lyucol}{rgb}{0.5, 0.1, 0.1}
\definecolor{gray}{rgb}{0.5, 0.5, 0.5}

\usepackage[left=2cm,right=2cm,top=2cm,bottom=2cm]{geometry}
\usepackage{amssymb}
\usepackage{hyperref}
\makeatletter

\newlength{\dhatheight}

\usepackage[normalem]{ulem}
\usepackage{dsfont}
\usepackage{accents}
\usepackage{verbatim}
\usepackage{ulem}
\usepackage{pgf,tikz,pgfplots}
\pgfplotsset{compat = 1.17}%take out before submission
\usepackage[all]{xy}

\newcommand{\eps}{\varepsilon}

\newcommand{\e}{\mathrm{e}}
\newcommand{\Po}{\mathrm{Po}}

\renewcommand{\Pr}{\mathbb{P}}
\newcommand{\cV}{\mathcal{V}}
\newcommand{\cC}{\mathcal{C}}
\newcommand{\cA}{\mathcal{A}}

\newcommand{\cE}{\mathcal{E}}

\newcommand{\cS}{\mathcal{S}}
\newcommand{\cO}{\mathcal{O}}

\usepackage{subfig}

\usepackage{enumerate}

\title{\scshape
Color-avoiding percolation on the Erd\H{o}s-R\'enyi random graph}

\author[1,2]{Lyuben Lichev}
\author[2,3]{Bruno Schapira}
\affil[1]{Univ. Jean Monnet, Saint-Etienne, France}
\affil[2]{Univ. Claude Bernard Lyon 1, Lyon, France}
\affil[3]{Univ. Aix-Marseille, Marseille, France}

\begin{document}

\maketitle
 
\begin{abstract}
We consider a recently introduced model of color-avoiding percolation (abbreviated CA-percolation) defined as follows. Every edge in a graph $G$ is colored in some of $k\ge 2$ colors. Two vertices $u$ and $v$ in $G$ are said to be \emph{CA-connected} if $u$ and $v$ may be connected using any subset of $k-1$ colors.  CA-connectivity defines an equivalence relation on the vertex set of $G$ whose classes are called \emph{CA-components}.

We study the component structure of a randomly colored Erd\H{o}s-R\'enyi random graph of constant average degree. We distinguish three regimes for the size of the largest component: a supercritical regime, a so-called intermediate regime, and a subcritical regime, in which the largest CA-component has respectively  linear, logarithmic, and bounded size. Interestingly, in the subcritical regime, the bound is deterministic and given by the number of colors.
\end{abstract}
\noindent
MSC Class: 05C80, 60C05, 60K35, 82B43

\section{Introduction}

In this paper, we are interested in the model of (edge-)color-avoiding percolation defined as follows. Fix a set of $k\ge 2$ colors and a graph $G$, and color every edge of $G$ in at least one color. We say that two vertices $u$ and $v$ in $G$ are \emph{color-avoiding connected}, or \emph{CA-connected} for short, if $u$ and $v$ may be connected using any subset of $k-1$ colors. In fact, CA-connectivity defines an equivalence relation on the vertex set of $G$ whose classes are called \emph{CA-components}. The model has been motivated by a number of real-world applications, for example, avoiding a set of mistrusted information channels (where colors correspond to eavesdroppers), or avoiding a set of possibly corrupted links in a network. In a sense, a network with large CA-components may be considered resistant to attacks from a set of adversaries where the adversaries control all channels but can only attack the network separately.

Color-avoiding percolation was introduced by Krause, Danziger, and Zlati\'c~\cite{KDZ16, KDZ17}. In their work, the authors were interested in vertex-colored graphs and analyzed a vertex analog of CA-connectivity. While some empirical observations were made for scale-free networks, the focus was put on Erd\H{o}s-R\'enyi random graphs due to their better CA-connectivity~\cite{KDZ16}. In a subsequent work, Kadovi\'c, Krause, Caldarelli, and Zlati\'c~\cite{KKCZ18} defined mixed CA-percolation where both vertices and edges have colors. To a large extent, each of these foundational papers based their conclusions on experimental evidence.

Precise mathematical treatment of the subject is challenging for several reasons. Firstly, unlike connected components in a graph, the CA-components cannot be found by a local exploration of the graph in general. Indeed, note that even if two vertices are neighbors in the graph, all paths that connect them and avoid a certain color may be rather long. Secondly, while one edge in a graph may merge at most two components or divide a single connected component into two parts, a single colored edge may lead to a merging of a lot of different CA-components. For example, consider two parallel paths of length $2\ell+1$, and for every $i\in [2\ell+1]$, connect the $i$-th vertex in one of the paths with the $i$-th vertex in the other path. Also, color the odd edges in the paths in blue, the even edges in red, and the edges between them in green, see Figure~\ref{fig 1}. Then, it may be easily checked that all CA-components in the obtained graph are of size 1 while adding a blue edge between the last vertices in the two paths creates $2\ell+1$ components of size 2 at once. Last but not least, in the framework of random graphs, deleting the edges in two different colors from the original graph leads to two distinct subgraphs that can have a large intersection, and can therefore be highly correlated.

\begin{figure}
\centering
\begin{tikzpicture}[line cap=round,line join=round,x=1cm,y=1cm]
\clip(-9.419216033144751,1.85) rectangle (6.330706009212033,3.15);
\draw [line width=0.8pt,dotted] (-8,3)-- (-7,3);
\draw [line width=0.8pt,dash pattern=on 4pt off 4pt] (-7,3)-- (-6,3);
\draw [line width=0.8pt,dotted] (-6,3)-- (-5,3);
\draw [line width=0.8pt,dash pattern=on 4pt off 4pt] (-5,3)-- (-4,3);
\draw [line width=0.8pt,dotted] (-4,3)-- (-3,3);
\draw [line width=0.8pt,dotted] (-8,2)-- (-7,2);
\draw [line width=0.8pt,dash pattern=on 4pt off 4pt] (-7,2)-- (-6,2);
\draw [line width=0.8pt,dotted] (-6,2)-- (-5,2);
\draw [line width=0.8pt,dash pattern=on 4pt off 4pt] (-5,2)-- (-4,2);
\draw [line width=0.8pt,dotted] (-4,2)-- (-3,2);
\draw [line width=0.8pt] (-8,3)-- (-8,2);
\draw [line width=0.8pt] (-7,3)-- (-7,2);
\draw [line width=0.8pt] (-6,3)-- (-6,2);
\draw [line width=0.8pt] (-5,3)-- (-5,2);
\draw [line width=0.8pt] (-4,3)-- (-4,2);
\draw [line width=0.8pt,dotted] (0,3)-- (1,3);
\draw [line width=0.8pt,dash pattern=on 4pt off 4pt] (1,3)-- (2,3);
\draw [line width=0.8pt,dotted] (2,3)-- (3,3);
\draw [line width=0.8pt,dash pattern=on 4pt off 4pt] (3,3)-- (4,3);
\draw [line width=0.8pt,dotted] (4,3)-- (5,3);
\draw [line width=0.8pt,dotted] (0,2)-- (1,2);
\draw [line width=0.8pt,dash pattern=on 4pt off 4pt] (1,2)-- (2,2);
\draw [line width=0.8pt,dotted] (2,2)-- (3,2);
\draw [line width=0.8pt,dash pattern=on 4pt off 4pt] (3,2)-- (4,2);
\draw [line width=0.8pt,dotted] (4,2)-- (5,2);
\draw [line width=0.8pt] (0,3)-- (0,2);
\draw [line width=0.8pt] (1,3)-- (1,2);
\draw [line width=0.8pt] (2,3)-- (2,2);
\draw [line width=0.8pt] (3,3)-- (3,2);
\draw [line width=0.8pt] (4,3)-- (4,2);
\draw [line width=0.8pt,dotted] (5,3)-- (5,2);
\draw [rotate around={90:(0,2.5)},line width=0.3pt] (0,2.5) ellipse (0.6099019513592782cm and 0.34925691155917815cm);
\draw [rotate around={90:(1,2.5)},line width=0.3pt] (1,2.5) ellipse (0.6099019513592774cm and 0.3492569115591777cm);
\draw [rotate around={90:(2,2.5)},line width=0.3pt] (2,2.5) ellipse (0.6099019513592759cm and 0.3492569115591768cm);
\draw [rotate around={90:(3,2.5)},line width=0.3pt] (3,2.5) ellipse (0.6099019513592845cm and 0.3492569115591819cm);
\draw [rotate around={90:(4,2.5)},line width=0.3pt] (4,2.5) ellipse (0.6099019513592845cm and 0.3492569115591819cm);
\begin{scriptsize}
\draw [fill=black] (-8,3) circle (1.5pt);
\draw [fill=black] (-7,3) circle (1.5pt);
\draw [fill=black] (-6,3) circle (1.5pt);
\draw [fill=black] (-5,3) circle (1.5pt);
\draw [fill=black] (-4,3) circle (1.5pt);
\draw [fill=black] (-3,3) circle (1.5pt);
\draw [fill=black] (-8,2) circle (1.5pt);
\draw [fill=black] (-7,2) circle (1.5pt);
\draw [fill=black] (-6,2) circle (1.5pt);
\draw [fill=black] (-5,2) circle (1.5pt);
\draw [fill=black] (-4,2) circle (1.5pt);
\draw [fill=black] (-3,2) circle (1.5pt);
\draw [fill=black] (0,3) circle (1.5pt);
\draw [fill=black] (1,3) circle (1.5pt);
\draw [fill=black] (2,3) circle (1.5pt);
\draw [fill=black] (3,3) circle (1.5pt);
\draw [fill=black] (4,3) circle (1.5pt);
\draw [fill=black] (5,3) circle (1.5pt);
\draw [fill=black] (0,2) circle (1.5pt);
\draw [fill=black] (1,2) circle (1.5pt);
\draw [fill=black] (2,2) circle (1.5pt);
\draw [fill=black] (3,2) circle (1.5pt);
\draw [fill=black] (4,2) circle (1.5pt);
\draw [fill=black] (5,2) circle (1.5pt);
\end{scriptsize}
\end{tikzpicture}
\caption{Blue, red and green edges are represented by dotted, dashed, and solid segments, respectively. One may easily check that on the left, every vertex is alone in its CA-component, while on the right, the addition of a single blue edge leads to the appearance of many CA-components of size 2.} 
\label{fig 1}
\end{figure}
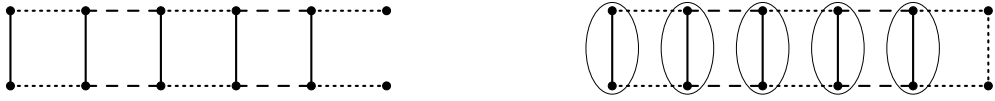

To our knowledge, a rigorous mathematical treatment of the behavior of color-avoiding percolation on random graphs has only been studied in a recent work of Fekete, Molontay, R\'ath and Varga~\cite{RVFM22}. In their paper, they show that under a certain subcriticality assumption, the number of CA-components of a given fixed size renormalized by $n$ converges in probability to a fixed constant. Moreover, under the same assumption, it is proved that the size of the largest CA-component renormalized by $n$ converges in probability to a fixed constant. They also characterize the behavior of that constant in the barely supercritical regime.

Our goal here is to go further in the analysis of the structure of CA-components in a randomly colored Erd\H{o}s-R\'enyi random graph around the threshold of appearance of a giant component. 
Apart from simplifying the approach of~\cite{RVFM22} (or, more precisely, the earlier version of this work~\cite{RVFM22.0}) and getting rid of their additional hypothesis, we show that the parameter space may be naturally divided into three regimes. In each of them, we conduct a careful analysis of the size of the largest CA-component as well as the number of small CA-components.

\subsection{Main results} 
For a positive integer $m$, we denote $[m] = \{1,\ldots,m\}$. In particular, we reserve the notation $[k]$ to denote the set of colors, and the notation $[n]$ for the set of vertices of our graphs. 
Recall that for $p\in [0,1]$, the Erd\H{o}s-R\'enyi random graph $G(n,p)$, or ER random graph for short, with parameters $n$ and $p$ is the graph on the vertex set $[n]$ where the edge between any two distinct vertices is present with probability $p$, independently from all other edges.   

Consider now a non-increasing sequence of positive real numbers $\lambda_1\ge \dots \ge \lambda_k$, and define a family of $k$ independent Erd\H{o}s-R\'enyi random graphs $G_i= G(n,\tfrac{\lambda_i}{n})$ for $i\in [k]$ on the same vertex set $[n]$ (alternatively, this family can be seen as a multigraph $([n],E_1,\dots,E_k)$ where $E_i$ is the edge set of $G_i$). In order to easily refer to and distinguish the graphs $(G_i)_{i=1}^k$, we say that for every $i\in [k]$, the edges of $G_i$ are given color $i$. Define further
$$\Lambda = \lambda_1 + \dots +\lambda_k,\quad \text{and}\quad
\lambda_i^* = \Lambda - \lambda_i \quad \text{for every }i\in[k].$$
In particular, $\lambda_1^* \le \lambda_2^*\le \dots \le \lambda_k^*$. Also, we set 
\begin{equation}\label{eq:defG}
G = G_1\cup \ldots\cup G_k,  
\end{equation}
and for $I\subseteq [k]$, 
$$G_I= \bigcup_{i\in I} G_i, \quad \text{and}\quad G^I = \bigcup_{i\in [k]\setminus I} G_i,$$
with the shorthand notation $G_i$ and $G^i$, respectively, when $I=\{i\}$.

Recall that two vertices $u$ and $v$ in $G$ are CA-connected if $u$ and $v$ are connected in each of the graphs $(G^{i})_{i=1}^k$. Moreover, CA-connectivity is an equivalence relation with classes called CA-components. The CA-component of a vertex $u\in [n]$ is denoted by $\widetilde \cC(u)$, and $|\widetilde \cC(u)|$ denotes its size. Our main object of interest is the size of the largest CA-component in $G$. Under the assumption that $\sum_{i\in I} \lambda_i<1$ for every set $I\subseteq [k]$ of size $k-2$, it was shown in~\cite{RVFM22.0} that there exists a constant $a\in [0,1]$ such that
\begin{equation}\label{eq:LLN.supercritical}
\frac {\max_{u\in [n]}|\widetilde \cC(u)|}n \xrightarrow[n\to \infty]{\Pr} a,
\end{equation}
and that the constant $a$ is positive if and only if $\lambda_1^*>1$, which is called the \emph{supercritical regime}. 
Here, we improve this result in three directions:
\begin{itemize}
\item We observe that~\eqref{eq:LLN.supercritical} can be easily derived from the convergence in distribution of the neighborhood of a typical vertex, and this does not require any additional technical assumptions.  
\item In the subcritical regime (that is, when $\lambda_k^*<1$), we prove that asymptotically almost surely (a.a.s.), any CA-component has size at most $k$. Moreover, the random variables $(N_\ell)_{\ell=2}^{k}$ which count the  number of CA-components of size $2,\ldots,k$, respectively, jointly converge in distribution to independent Poisson random variables. 
\item We start investigating the more difficult intermediate regime (that is, when $\lambda_k^* > 1 > \lambda_1^*$). In this setting, we show a weak law of large numbers under the assumption that $\lambda_k^*>1>\lambda_{k-1}^*$. 
\end{itemize}
These results are summarized in the following theorem. 

\begin{theorem}\label{thm:k>2}
Suppose that $k\ge 2$.
\begin{enumerate}[(i)]
    \item\label{pt i} There exists $a_1\in [0,1)$,  such that
    \[\frac{\max_{u\in [n]}|\widetilde \cC(u)|}{n}\xrightarrow[n\to \infty]{\Pr} a_1.\]
    Moreover, one has $a_1>0$ if and only if  $\lambda_1^* > 1$.
    \item\label{pt ii} If $\lambda_k^* > 1 > \lambda_{k-1}^*$, then there is a positive constant $a_2$ such that
    \[\frac{\max_{u\in [n]}|\widetilde \cC(u)|}{\log n}\xrightarrow[n\to \infty]{\Pr} a_2.\]
    \item\label{pt iii} If $ \lambda_k^*<1$, then a.a.s.\ $\max_{u\in [n]}|\widetilde \cC(u)| \le k$. Moreover, there are positive constants $\beta_2, \ldots, \beta_k$, such that
    \[(N_2, \ldots, N_k)\xrightarrow[n\to \infty]{d} \bigotimes_{\ell=2}^k  \Po(\beta_\ell),\]
    that is, the random variables $(N_\ell)_{\ell=2}^k$ jointly converge in distribution to $k-1$ independent Poisson variables with means $\beta_2, \ldots, \beta_k$ as $n\to \infty$.
\end{enumerate}
\end{theorem}

Note that in the above result, we recover an analog of the famous transition from a linear to a logarithmic size for the largest connected component, which appears in an ER random graph $G(n,\tfrac{\lambda}{n})$ as the parameter $\lambda$ crosses the critical value $1$. However, unlike this standard setting, one original feature of CA-percolation is that there is an additional regime where the size of the largest CA-component remains bounded. Furthermore, the critical cases $\lambda_1^*=1$ and $\lambda_k^*=1$ appear to be more subtle, see Proposition~\ref{prop:critical} below and Section~\ref{sec:discussion} for more on this.

Now, we take a closer look at each of the three parts of Theorem~\ref{thm:k>2}. Part~\eqref{pt i} should not be very surprising as, on the one hand, if $\lambda_1^*\le 1$, then a.a.s.\ the largest connected component in $G^1$ has sublinear size, and thus the largest CA-component as well. On the other hand, if $\lambda_1^*>1$, then a.a.s.\ the largest CA-component is obtained by intersecting the largest connected components in each of the graphs $G^1,\dots,G^k$, which all have linear size. However, we stress that the whole point of the proof is to handle the lack of independence between these components. To do this, we use a local limit argument for a sequence of randomly colored ER random graphs, which also allows us to recover and improve in a simple way a result of~\cite{RVFM22.0,RVFM22} on the number of CA-components of given size. 
\begin{proposition}\label{prop:Nell}
There exists a sequence of non-negative real numbers $(\nu_\ell)_{\ell \ge 1}$ such that $\sum_{\ell \ge 1} \nu_\ell = 1-a_1$, and for each $\ell \ge 1$, 
\begin{equation}\label{eq:asympt.CA}
\frac{|\{u\in [n]: |\widetilde \cC(u)| = \ell\}|}{n} = \frac{\ell\cdot N_\ell}{n} \xrightarrow[n\to \infty]{\Pr} \mathbb \nu_\ell. 
\end{equation}
Moreover, one has $\nu_\ell >0$ for all $\ell \ge 1$ if and only if $\lambda_k^*>1$, while if $\lambda_k^*\le 1$, then $\nu_1=1$. 
\end{proposition}
Altogether, this proposition and Part~\eqref{pt i} of Theorem~\ref{thm:k>2} answer a question from~\cite{RVFM22.0} by showing that one can get rid of their technical assumption. 

For Part~\eqref{pt ii} of Theorem~\ref{thm:k>2}, the main observation is that the largest CA-component comes from intersecting a connected component in $G_k$ with the largest component of $G^k$. In fact this argument also gives us that the size of the largest CA-component in the critical case $\lambda_k^*=1>\lambda_{k-1}^*$ is tight.

\begin{proposition}\label{prop:critical}
Suppose that $\lambda_k^* = 1>\lambda_{k-1}^*$. Then,
    \[\sup_{n\ge 1}\ \Pr\Big(\max_{u\in [n]} |\widetilde\cC(u)|\ge M\Big) \xrightarrow[M\to \infty]{} 0.\]
\end{proposition}
The question of whether in this critical case the size of the largest CA-component converges in distribution remains open. In fact, even knowing if the support is asymptotically bounded by a deterministic constant is unknown.

In the remaining intermediate regime, i.e. when $\lambda_{k-1}^* \ge 1 > \lambda_1^*$, one can easily show that the size of the largest CA-component is still of logarithmic order. However, proving concentration seems to be a challenging problem, which remains out of reach with our present techniques.

Part~\eqref{pt iii}, as we already mentioned, is arguably the most original part. In this case, the largest connected component in each of the graphs $G^1,\dots,G^k$ has only logarithmic size, which makes it very difficult for a pair of vertices to be connected in all $k$ graphs. This explains, at least heuristically, why the largest CA-component has bounded size, although the fact that the bound is deterministic may appear as unexpected. Actually, as we shall see, the main reason for which the upper bound is given by the number of colors is that a.a.s.\ every CA-component is either contained in a single edge, or all paths connecting two of its vertices and avoiding some color are contained in a single cycle.

Finally, we provide an explicit expression of $\beta_k$ in terms of 
$\lambda_1,\ldots,\lambda_k$ in Remark~\ref{rem:betas}. While it is possible to do the same for the other values of $\beta_m$ for $m<k$, the formula and the computation tend to get more and more tedious as $m$ decreases.

\vspace{0.2cm}
We now comment on the proofs themselves in more detail.

\paragraph{Outline of the proofs.} As already mentioned, the proofs of Part~\eqref{pt i} of Theorem~\ref{thm:k>2} and Proposition~\ref{prop:Nell} are based on the well-known local convergence of Erd\H{o}s-R\'enyi random graphs to Bienaym\'e-Galton-Watson trees, or BGW trees for short, with Poisson offspring distribution.
In our setting of edge-colored ER random graph, the local limit can be seen as a BGW tree with $\text{Po}(\Lambda)$ offspring distribution where additionally each edge is colored with color $i\in [k]$ with probability $\lambda_i/\Lambda$, independently for different edges. Then, the constant $a_1$ is just the probability that for every $i\in [k]$, the connected component of the root is infinite when we erase the edges colored with color $i$. The same approach allows to prove Proposition~\ref{prop:Nell}, with  $(\nu_\ell)_{\ell \ge 1}$ being the probability distribution of the size of the CA-component of the root in the aforementioned BGW tree. 
However, here the notion of CA-component must be suitably adjusted, as already noticed in~\cite{RVFM22.0}: for each $i\in [k]$, two vertices are declared to be connected when erasing color $i$ if either they are indeed connected by a path in the BGW tree that avoids color $i$, or if they both connect to infinity by paths avoiding color $i$. It is then not difficult to see that, for this notion of CA-connectivity, the constant $a_1$ is also the probability that the CA-component of the root is infinite.

On the other hand, the proofs of~\eqref{pt ii}, Proposition~\ref{prop:critical} and~\eqref{pt iii} have a different flavor. Firstly, we outline the proof strategy for~\eqref{pt ii}. To begin with, we make use of the following fact, which might be of independent interest. Consider a subcritical Erd\H{o}s-R\'enyi random graph $G(n,\tfrac{\lambda}{n})$, and independently color each of its vertices in black with some probability $q\in (0,1]$. Then, we show that the maximal number of black vertices which are all connected in $G(n,\tfrac{\lambda}{n})$ divided by $\log n$ converges in probability towards a positive constant. This result extends the concentration for the size of the largest component in a subcritical ER random graph, which corresponds to $q=1$. The proof is obtained by solving an energy-entropy optimization problem, where the energy corresponds to the cost of having a large number of black vertices in a given connected component of $G(n,\tfrac{\lambda}{n})$, while the entropy factor comes from the fact that as $b$ decreases, the number of connected components of size $b\log n$ in $G(n,\tfrac{\lambda}{n})$ increases. The link with~\eqref{pt ii} arises as one realizes (the nontrivial fact) that when $\lambda_{k-1}^*<1$, the largest CA-component necessarily comes from intersecting a component in $G_k$ (which is subcritical) with the largest component of $G^k$ (which is supercritical), the two being independent, and the fact that the vertex set of the giant component in the latter can be approximated by a binomial random subset of vertices. A similar  argument also leads to a proof of Proposition~\ref{prop:critical}. We refer to Sections~\ref{sec:stoc.dom},  \ref{sec.prel.intersection}, and~\ref{sec:proof.ii} for more details.

Concerning~\eqref{pt iii}, the crux of the proof is to show that when $\lambda_k^*<1$, a.a.s.\ every CA-component of size at least two is contained either in a cycle of $G$ or in a single edge (the latter happening only if the CA-component has exactly two vertices which are linked by an edge in at least two different colors), see Lemma~\ref{cor:support.cycle}. The proof of this result relies on a combination of some counting arguments (e.g. showing that a.a.s.\ connected subgraphs of $G(n,\lambda/n)$ of size at most some constant times $\log n$ contain at most one cycle, see Lemma~\ref{lem:2 cycles}), and a more probabilistic lemma showing that a.a.s.\ all CA-components have bounded size (see in particular  Lemma~\ref{lem:Ijconnect}). Once this is established,~\eqref{pt iii} follows from standard results on the asymptotic numbers of short cycles in an Erd\H{o}s-R\'enyi random graph.

\begin{remark}
Note that the notion of CA-component of a vertex $u$ is usually defined as the set of vertices connected to $u$ when deleting edges of color $i$ for any $i\in [k]$ (that is, unlike in our setting, edges in color $i$ cannot be used even if they have more than one color). While all our results would also hold with this definition, ours is slightly more convenient to deal with, especially in the intermediate regime, since this way $G_i$ and $G^i$ are independent.
\end{remark}

\paragraph{Further notation and terminology.} 
In general, we omit the dependence on $n$, $\Lambda$ and $(\lambda_i)_{i\in [k]}$ for convenience of notation.

For $I\subseteq [k]$ and a vertex $u\in [n]$, we denote by $\cC_I(u)$ and $\cC^I(u)$ the connected components of $u$ in $G_I$ and $G^I$, respectively. An edge is said to be \emph{repeated} if it participates in at least two of $G_1, \ldots, G_k$.

In this paper, we often identify graphs with their vertex sets. For instance, given a graph $H$, the \emph{size} of $H$ stands for the number of vertices in $H$, which we denote by $|H|$.
%we call \emph{size} of $H$ its number of vertices, which we denote by $|H|$. 
For a set $S$ of vertices of $H$, we denote by $H[S]$ the subgraph of $H$ induced by $S$ (that is, the graph with vertex set $S$ and edge set given by the edges of $H$ with both vertices in $S$). 
Moreover, for two vertices $u,v$ in $H$, we denote by $\{u\xleftrightarrow{H\,} v\}$ the event that $u$ and $v$ are connected in $H$.

A sequence of events $(\cE_n)_{n\ge 1}$ is said to hold a.a.s.\ if $\mathbb P(\cE_n)\to 1$ as $n\to \infty$. Given two positive real sequences $(f_n)_{n\ge 1}$ and $(g_n)_{n\ge 1}$, we write $f_n=o(g_n)$ if $f_n/g_n\to 0$ when $n\to\infty$, and $f_n=\mathcal O(g_n)$ if there exists a constant $C>0$ such that $f_n\le Cg_n$ for all $n\ge 1$. Furthermore, we write $\xrightarrow{d}$ to denote convergence in distribution of a sequence of random variables, and $\xrightarrow{\mathbb P}$ for convergence in probability.

Finally, we denote by $\Po(\lambda)$ the Poisson distribution with parameter $\lambda$, and by $\text{Bin}(n,q)$ the Binomial distribution with parameters $n$ and $q$. For a family of distributions $(\mu_i)_{i\in I}$, we denote by $\bigotimes_{i\in I} \mu_i$ the distribution of a vector $(X_i)_{i\in I}$ of independent random variables where $X_i\sim \mu_i$ for every $i\in I$.

\paragraph{Plan of the paper.} The rest of the paper is organized as follows. In the next section, we recall some known facts about ER random graphs, both in the subcritical and supercritical regimes, and show the result on the size of the largest black connected set in a (vertex-)colored subcritical ER random graph which was mentioned above. 
Then, in Section~\ref{sec:proof.theo} we give the proof of our main results, Theorem~\ref{thm:k>2} and Propositions~\ref{prop:Nell} and~\ref{prop:critical}, and finally discuss some open questions in Section~\ref{sec:discussion}. 

\section{\texorpdfstring{Preliminaries on Erd\H{o}s-R\'enyi random graphs}{}}
In this section, we gather some results on the random graph $G(n,\tfrac \lambda n)$. Apart from the results in Section~\ref{sec.prel.intersection}, most of the material presented here is well-known, and we sometimes include short proofs only for the reader's convenience. 

We let $\cC(u)$ denote the connected component of a vertex $u \in [n]$. 
Also, for $\lambda>0$, set 
\begin{equation}\label{def.Ilambda}
I_\lambda = \lambda- 1 - \log \lambda.
\end{equation} 
It is easy to check that this is a positive real number for all $\lambda>0$ different from 1. Also, for every integer $s\ge 0$, we define 
$$Z_s = \sum_{u\in [n]} \mathds{1}_{\{|\cC(u)|\,\ge\, s\}}. $$

\subsection{Subcritical regime: cluster size and two-point connectivity} 
Fix $\lambda\in (0,1)$. Then, it is well-known that
\begin{equation}\label{LLN.subcritical}
\frac{\max_{u\in [n]} |\cC(u)|}{\log n} \xrightarrow[n\to \infty]{\mathbb P} \frac 1{I_\lambda},
\end{equation}
see e.g. Theorems 4.4 and 4.5 in~\cite{Hof16}. 
It will also be useful to have a bound on the upper tail of the typical cluster size. The following one will be sufficient for our purposes.

\begin{lemma}[see display (4.3.11) in~\cite{Hof16}]\label{lem:sizeERC} Fix $\lambda\in (0,1)$. Then, for every $t\ge 1$, 
$$\mathbb P(|\cC(u)|\ge t) \le \e^{- I_\lambda\cdot\, t}.$$
\end{lemma}

A reverse inequality holds as well when $t$ does not grow too fast with $n$. In particular, we shall need the following result (see e.g. displays (4.3.34) and (4.3.37) in~\cite{Hof16}).
\begin{lemma}\label{lem:reverse.ERC}
Fix $\lambda\in(0,1)$ and $a\in (0,1/I_\lambda]$. Then,  
$$\mathbb P(|\cC(u)|\ge a\log n) \ge n^{-(1+o(1))I_\lambda\cdot\, a}.$$
\end{lemma}

Moreover, Lemma~\ref{lem:sizeERC} has the following important consequence. 

\begin{corollary} \label{cor:connectionER}
Fix $\lambda\in (0,1)$. Then, for every pair of distinct vertices $u,v$ in $G(n,\frac \lambda n)$, 
$$\mathbb P(v\in \cC(u)) \le \frac 1n \cdot \frac {\e^{-I_\lambda}}{1-\e^{-I_\lambda}}.$$
\end{corollary}
\begin{proof}
It suffices to observe that for every $t\ge 1$, conditionally on the event $\{|\cC(u)|=t\}$, the set of vertices different from $u$ and contained in $\cC(u)$ is uniformly distributed among all possible subsets of $[n]\setminus \{u\}$ of size $t-1$. In particular, for every $v\neq u$, 
\begin{equation}\label{prob.cond.connect}
\mathbb P(v\in \cC(u) \mid |\cC(u)|= t) = \frac{t-1}{n-1}\le \frac{t}{n}.
\end{equation}
By summing over all positive integers $t\in [n]$ and using Lemma~\ref{lem:sizeERC}, we get 
$$\mathbb P(v\in \cC(u)) \le \sum_{t = 1}^n \frac{t}{n}\cdot  \mathbb P(|\cC(u)|= t) = \frac 1n \sum_{t = 1}^n\mathbb P(|\cC(u)|\ge t) \le \frac 1n \cdot \frac {e^{-I_\lambda}}{1-e^{-I_\lambda}},$$
as desired.
\end{proof}

The next result provides concentration of the variables $Z_s$ when $s$ is of order $\log n$. The equality for the expectation is a direct consequence of Lemmas~\ref{lem:sizeERC} and~\ref{lem:reverse.ERC}, while the two inequalities for the variance follow from Proposition 4.7 in~\cite{Hof16} and Lemma~\ref{lem:sizeERC}, respectively.

\begin{lemma}\label{lem:clustersize}
Let $\lambda\in (0,1)$ and $a\in (0,1/I_\lambda]$. Then, 
$$\mathbb E[Z_{a\log n}] = n^{1-(1+o(1)) I_\lambda\cdot\, a} \quad \text{and}\quad \mathrm{Var}[Z_{a\log n}] \le n\cdot \mathbb E\big[|\cC(1)|\mathds{1}_{\{|\cC(1)|\,\ge\, a\log n\}}\big]\le n^{1-(1+o(1)) I_\lambda\cdot\, a}.$$
\end{lemma}

\subsection{On the number of cycles}

We start with a result showing that a.a.s.\ all connected subgraphs of $G(n, \tfrac{\lambda}{n})$ of size at most some constant (depending on $\lambda$) times $\log n$ contain at most one cycle, and when $\lambda<1$, all those with size at least $\varepsilon \log n$ are trees for any fixed $\varepsilon>0$.

\begin{lemma}\label{lem:2 cycles}
Fix $\lambda > 0$. 
\begin{enumerate}
\item There is a positive constant $c_1=c_1(\lambda)$ such that a.a.s.\ the following holds: for every set $ S\subseteq [n]$ with $|S| \le c_1 \log n$, whenever the subgraph of $G(n,\tfrac{\lambda}{n})$ induced by $S$ is connected, it contains at most one cycle. 
\item If $\lambda < 1$, then for every $\varepsilon>0$, a.a.s.\ all components of size larger than $\varepsilon \log n$ are trees.
\end{enumerate}
\end{lemma}
\begin{remark} In fact, concerning the second statement of this lemma, more is true: as shown in the proof below, the expected number of vertices whose connected component contains at least one cycle is bounded uniformly in $n$. 
\end{remark}
\begin{proof}[Proof of Lemma~\ref{lem:2 cycles}]
Concerning the first part of the lemma, note that any connected graph with at least two cycles contains as a subgraph a spanning tree together with two additional edges. Furthermore, by Cayley's formula the number of spanning trees of the complete graph with size $m$ is equal to $m^{m-2}$. Therefore, for every $c>0$, by using that $\tbinom{n}{m}\le \left(\tfrac{n\e}{m}\right)^m$, 
we deduce that the expected number of subgraphs of $G(n,\tfrac{\lambda}{n})$ which are connected, contain at least two cycles and at most $c\log n$ vertices, is bounded from above by, 
\begin{equation*}
\sum_{m=1}^{\lfloor c\log n\rfloor} m^{m-2}\cdot \binom{m}{2}^2\cdot \binom{n}{m}\cdot  \left(\frac{\lambda}{n}\right)^{m+1} \le \frac{1}{n}\sum_{m=1}^{\lfloor c\log n\rfloor} m^2 (\e\lambda)^{m+1},
\end{equation*}
which is $o(1)$ if one chooses  $c<\frac{1}{1+\max(0,\log \lambda)}$. The proof of the first part is completed by an application of Markov's inequality.

For the second part, we compute the expected number of vertices in components of size at most $2 I_{\lambda}^{-1} \log n$ containing a cycle. Taking also into account the fact that no vertex in a component of size $m$ is connected to any of the $n-m$ vertices outside the component, the previous computation implies that the above expectation is at most
\[\sum_{m=1}^{2 I_{\lambda}^{-1} \log n} m\cdot m^{m-2}\cdot \binom{m}{2}\cdot \binom{n}{m} \cdot  \left(\frac{\lambda}{n}\right)^m \left(1-\frac{\lambda}{n}\right)^{m(n-m)} \le (1+o(1))\sum_{m=1}^{2 I_{\lambda}^{-1} \log n} m\cdot \e^{-I_{\lambda} m} = \cO(1).\]
Thus, by Markov's inequality there are a.a.s.\ less than $\varepsilon \log n$ vertices in components of size at most $2 I_{\lambda}^{-1} \log n$ containing a cycle. However, by Lemma~\ref{lem:sizeERC} a.a.s.\ all components have size at most $2 I_{\lambda}^{-1} \log n$, which completes the proof.
\end{proof}

\begin{remark}\label{rem:2 cycles}
The following modification of Lemma~\ref{lem:2 cycles} holds with almost the same proof.
\begin{enumerate}
\item There exists $c_1=c_1(\Lambda)$ such that a.a.s.\ for every set $S\subseteq [n]$ with $|S| \le c_1 \log n$, whenever the subgraph of $G$ induced by $S$ is connected, it either contains at most one cycle and no repeated edges or at most one repeated edge and no cycles. 
\item If $\lambda_i^*<1$ for some $i$, then for every $\varepsilon>0$, a.a.s.\ all connected components of $G^i$ of size at least $\varepsilon \log n$ are trees with no repeated edges.
\end{enumerate}
\end{remark}

The next lemma is a well-known result concerning the number of cycles of given size that will be needed for the proof of Theorem~\ref{thm:k>2}~\eqref{pt iii}. We refer e.g. to Corollary 4.9 from~\cite{Bol01} for a proof. 

\begin{lemma}[\cite{Bol01}, Corollary 4.9]\label{lem:Bol}
Fix $\lambda > 0$. For $m\ge 3$, denote by $C_m$ the number of cycles of length $m$ in the graph $G(n,p)$ with $p =  (1+o(1))\tfrac{\lambda}{n}$. Then, for any fixed $\ell \ge 3$, 
\[(C_3, \ldots, C_{\ell})\xrightarrow[n\to \infty]{d} \bigotimes_{m=3}^\ell \Po(\gamma_m),\]
where for all $m\in \{3,\dots,\ell\}$, $\gamma_m = \tfrac{\lambda^m}{2m}$. 
\end{lemma}

Lemma~\ref{lem:Bol} has the following direct consequence. Recall that an edge is said to be repeated in $G$ if it is part of at least two of the graphs $G_1,\dots,G_k$. 

\begin{corollary}\label{cor:Bol}
Denote by $C_2$ the number of repeated edges in $G$. Then, with the same notation as in  Lemma~\ref{lem:Bol}, for every $\ell\ge 2$,
\[(C_2, \ldots, C_{\ell})\xrightarrow[n\to \infty]{d} \bigotimes_{m=2}^\ell \Po(\gamma_m),\]
where $\gamma_2 = \tfrac{1}{2} \sum_{i,j\in [k], i<j} \lambda_i \lambda_j$ and for all $m\in \{3,\dots,\ell\}$, $\gamma_m = \tfrac{\Lambda^m}{2m}$.
\end{corollary}
\begin{proof}
Using that $G$ is distributed as an Erd\H{o}s-R\'enyi random graph with parameters $n$ and 
\[p = 1-\prod_{i=1}^k \left(1-\frac{\lambda_i}{n}\right) =(1+o(1))\cdot \frac{\Lambda}{n},\]
the joint convergence of $(C_3, \ldots, C_\ell)$ is given by Lemma~\ref{lem:Bol}. 

On the other hand, by definition $C_2\sim\text{Bin}(\tfrac{n(n-1)}{2},\tfrac{2(1+o(1))\gamma_2}{n^2})$, and thus $C_2$ converges in distribution to $\text{Po}(\gamma_2)$. It only remains to justify the asymptotic independence between $C_2$ and the other variables. The argument is the same as the one showing asymptotic independence of $C_3,\dots,C_\ell$, see Theorem~4.8 and Corollary~4.9 in~\cite{Bol01}. Briefly, one can first notice by a simple first moment argument that a.a.s.\ no vertex participates simultaneously in a repeated edge and in a cycle of length at most $\ell$. Thus, the variables $(C_3,\dots,C_\ell)$ a.a.s.\ coincide with the cycle counts in the graph, obtained by deleting all vertices in repeated edges, which conditionally on $C_2$ is an ER random graph with at least $n-2C_2$ vertices. Using that $\mathbb E[C_2]=\mathcal O(1)$, and more precisely that a.a.s.\ $n-2C_2 = n -o(n)$, implies the corollary.
\end{proof}

\subsection{Supercritical regime: stochastic domination of the giant component}\label{sec:stoc.dom}
Recall that when $\lambda>1$, the graph $G(n,\tfrac \lambda n)$ has a.a.s.\ a unique connected component of linear size (called the \emph{giant} component). More precisely, it is well-known that 
\begin{equation}\label{LLN.giant}
\frac{\max_{u\in [n]} |\cC(u)|}{n} \xrightarrow[n\to\infty]{\mathbb P} \mu_\lambda
\end{equation}
where $\mu_\lambda$ is the survival probability of a Bienaym\'e-Galton-Watson tree with offspring distribution $\mathrm{Po}(\lambda)$ characterized as the unique positive solution of the equation $1=e^{-\lambda t} + t$, see display (3.6.2) and Theorem 4.8 in~\cite{Hof16}. Recall also that conditionally on its size, the set of vertices in the giant component is uniformly distributed among the family of subsets of $[n]$ of that size. As a consequence, one has the following stochastic comparison with binomial random subsets of vertices. 

\begin{lemma}\label{lem:stoch.giant}
Fix $\lambda>1$ and $\varepsilon>0$. Let $\cC_{\max}$ be the a.a.s.\ unique largest connected component of $G(n,\tfrac \lambda n)$. Let also $(X_v)_{v\in [n]}$ and $(Y_v)_{v\in [n]}$ be two sequences  
of i.i.d.\ Bernoulli random variables with respective parameters $\max(\mu_\lambda-\varepsilon,0)$ and $\min(\mu_\lambda+\varepsilon,1)$. Then, there is a coupling of these two sequences with $G(n,\tfrac{\lambda}n)$ such that a.a.s.\ one has
\begin{equation}\label{eq:inclusion.giant}
\{v : X_v=1\} \subseteq \cC_{\max} \subseteq \{v : Y_v=1\}.
\end{equation} 
\end{lemma}

\begin{proof} 
Note that conditionally on their respective sizes, the three sets appearing in~\eqref{eq:inclusion.giant} are sampled uniformly at random among all subsets of $[n]$ of that size. The lemma then follows by the weak law of large numbers and~\eqref{LLN.giant}, which together imply that a.a.s.\
$$\sum_{v\in [n]} X_v\le |\cC_{\max}|\le \sum_{v\in [n]} Y_v.$$
\end{proof}

\subsection{On the intersection of a giant with independent subcritical clusters}\label{sec.prel.intersection}
Fix $q\in (0,1)$ and $\lambda \in (0,1)$. For every $x\in [q,1)$, define
\[J_q(x)= x\log \frac xq + (1-x)\log \frac{1-x}{1-q},\quad \text{and set}\quad \rho(q,\lambda) = \inf_{x\in [q,1)}  \frac {I_\lambda+J_q(x)}{x}.\]
Consider a sequence $(X_v)_{v \in [n]}$ of independent Bernoulli random variables with parameter $q$, and independently a graph $G(n,\tfrac {\lambda} n)$. 
For $u\in [n]$ and $t\ge 0$, define 
$$\widetilde Z_t = \sum_{u\in [n]} \mathds{1}_{\{\sum_{v\in \cC(u)} X_v\, \ge\, t\}}. $$
The next lemma is similar in essence to Lemma~\ref{lem:clustersize}.
\begin{lemma}\label{lem:inter.giant.sub}
Fix $a\in (0,\tfrac{1}{\rho(q,\lambda)}]$. Then, 
$$\mathbb E[\widetilde Z_{a\log n}] = n^{1-(1+o(1)) \rho(q,\lambda)\,\cdot\, a } \quad \text{and}\quad \mathrm{Var}[\widetilde Z_{a\log n}] \le n^{1-(1+o(1)) \rho(q,\lambda)\,\cdot\, a}.$$ 
\end{lemma}
\begin{proof}
We begin by proving the estimate on the mean. To start with, we recall two large deviation estimates for Binomial random variables. On the one hand, for every $x\in [q,1)$ and every $N\ge 1$ (see e.g. Corollary 2.20 in \cite{Hof16}), \begin{equation}\label{eq:chernoff1}
\mathbb P(\text{Bin}(N,q)\ge xN) \le \exp\big(- N \cdot J_q(x)\big),
\end{equation}
and on the other hand, for every fixed $x\in [q,1)$ (see e.g. Theorem 2.2.3 and Exercice 2.2.23 in \cite{DZ09}),
\begin{equation}\label{eq:chernoff2}
\mathbb P(\text{Bin}(N,q)\ge xN) \ge  \exp\big(-(1+o_N(1))\cdot N \cdot J_q(x)\big),
\end{equation}
where the $o_N(1)$ term goes to $0$ as $N\to \infty$. 
Since conditionally on its size the cluster $\cC(1)$ is  uniformly distributed among the subsets of $[n]$ of that size containing the vertex 1, we deduce using~\eqref{eq:chernoff1} and Lemma~\ref{lem:sizeERC} that 
\begin{equation}
\begin{split}\label{eq:X_v}
\mathbb P\Big(\sum_{v\in \cC(1)} X_v\ge a\log n\Big) & \le \sum_{s=a\log n}^{(a/q) \log n} \mathbb P(|\cC(1)|=s) \cdot \mathbb P(\text{Bin}(s,q)\ge a\log n)+ \mathbb P(|\cC(1)|\ge (a/q)\cdot \log n)  \\
& \le \sum_{s=a\log n}^{(a/q)\log n} \exp\left(-\left(I_\lambda +J_q\left(\frac{a\log n}{s}\right)\right)  \cdot s\right)+ n^{-I_\lambda \cdot\, a/q} \\
&\le \mathcal O(\log n)\cdot n^{- \inf_{b\in (a,a/q]}  b(I_\lambda + J_q(a/b))}, 
\end{split}
\end{equation}
and observing that 
\begin{equation}\label{eq:infb}
\inf_{b\in (a,a/q]}  b(I_\lambda + J_q(a/b)) 
= a\cdot  \rho(q,\lambda),
\end{equation}
concludes the proof of the upper bound. 

Now, we concentrate on the lower bound. To start with, we extend $J_q$ to a continuous function on the interval $[q,1]$ by defining $J_q(1)= \log(1/q)$, and let $b_*$ be the smallest real number realizing the infimum of the function $b\mapsto b(I_\lambda + J_q(a/b))$ over the interval $[a,a/q]$. 
Recall that $\log$ is a concave function, so by Jensen's inequality the function $J_q$ is non-negative. Then, together with the fact that $a\le 1/\rho(q,\lambda)$ by our hypothesis,~\eqref{eq:infb} shows that
\[b_* = \frac{a \cdot \rho(q,\lambda)}{I_{\lambda} + J_q(a/b_*)} \le \frac{a \cdot \rho(q,\lambda)}{I_\lambda} \le \frac{1}{I_\lambda}.\]
Thus, by using Lemma~\ref{lem:reverse.ERC} and~\eqref{eq:chernoff2}, we get
\begin{align*} 
\mathbb P\Big(\sum_{v\in \cC(1)} X_v\ge a\log n\Big) 
&\ge\; 
\mathbb P(|\cC(1)|\ge b_*\log n) \cdot \mathbb P(\text{Bin}(b_*\log n,q)\ge a\log n)\\
& \ge \; n^{-(1+o(1))b_*\cdot I_{\lambda}}\cdot n^{-(1+o(1)) b_*\cdot J_q(a/b_*)}=  n^{-(1+o(1)) a\,\cdot\, \rho(q,\lambda)},
\end{align*}
which concludes the proof of the lower bound.

Finally, the proof of the upper bound on the variance is mutatis mutandis the same as the proof of Lemma~\ref{lem:clustersize}, in particular, the same argument leads to
$$\text{Var}[\widetilde Z_{a\log n}]\le n\cdot \mathbb E\left[\left(\sum_{v\in \cC(1)} X_v\right)\cdot \mathds{1}_{\{\sum_{v\in \cC(1)} X_v \,\ge\, a \log n\}} \right],$$
which combined with~\eqref{eq:X_v} yields the desired upper bound. 
\end{proof}

\begin{remark}\label{rem:alla} A close look at the previous proof shows that the upper bound on the mean is in fact valid for all $a>0$. 
\end{remark}

\begin{corollary}\label{cor:LLN.inter.giant}
Let $a=a(q,\lambda)= \frac{1}{\rho(q,\lambda)}>0$. Then,
$$\frac{\max_{u\in [n]} \sum_{v\in \cC(u)} X_v}{\log n} \xrightarrow[n\to\infty]{\mathbb P} a. $$
\end{corollary}
\begin{proof}
If $h>a(q,\lambda)$, then by Markov's inequality and Lemma~\ref{lem:inter.giant.sub} (see also Remark~\ref{rem:alla}) we get 
$$\mathbb P\left(\max_{u\in [n]} \sum_{v\in \cC(u)} X_v \ge h \log n \right) = \mathbb P(\widetilde Z_{h\log n}\ge 1) \le\mathbb E[\widetilde Z_{h\log n}] = o(1). $$ 
Conversely, if $h<a(q,\lambda)$, then again Lemma~\ref{lem:inter.giant.sub} together with the Cauchy-Schwarz inequality gives 
$$\mathbb P\left(\max_{u\in [n]} \sum_{v\in \cC(u)} X_v \ge h \log n \right) = \mathbb P(\widetilde Z_{h\log n}\ge 1) \ge 
\frac{\mathbb E[\widetilde Z_{h\log n}]^2}{\mathbb E[\widetilde Z_{h\log n}]^2+\text{Var}[\widetilde Z_{h\log n}]} = 1-o(1).$$
\end{proof}

\section{\texorpdfstring{Proofs of the main results}{}}
\label{sec:proof.theo}

This section is devoted to the proofs of Theorem~\ref{thm:k>2} and Proposition~\ref{prop:Nell}. We first prove Part~\eqref{pt i} of the theorem together with the proposition in Section~\ref{sec:proof.i}, then Part~\eqref{pt ii} of the theorem in 
Section~\ref{sec:proof.ii}, and finally Part~\eqref{pt iii} in Section~\ref{sec:proof.iii}.

\subsection{\texorpdfstring{Proofs of Theorem~\ref{thm:k>2}~\eqref{pt i} and Proposition~\ref{prop:Nell}}{}}
\label{sec:proof.i}
The proofs of these two results are based on the well-known local convergence of ER random graphs to BGW trees with Poisson offspring distribution. Recall that in the standard setting of uncolored graphs, the local \emph{weak} convergence states the following: denoting by $\mathcal V_L(u)$  the $L$-neighborhood of a vertex $u$ in $G$ (as defined in~\eqref{eq:defG})
or in the BGW tree with $\text{Po}(\Lambda)$ offspring distribution, for every bounded function $(u,H)\mapsto \varphi(u,H)$ defined on pairs $(u,H)$ where $H$ is a finite graph and $u$ a vertex of $H$, one has
$$\frac 1n \sum_{u\in [n]} \mathbb E[\varphi(u,\mathcal V_L(u))]\xrightarrow[n\to \infty]{} \mathbb E[\varphi(\varnothing,\mathcal V_L(\varnothing))],$$
see e.g.\ Theorem 6 in~\cite{Cur17}. It is straightforward to see that the same convergence holds in our setting of colored graphs. The difference is that now, edges are endowed with a label encoding its set of colors. As a consequence, the limiting graph is a BGW tree with $\text{Po}(\Lambda)$ offspring distribution, denoted hereafter by $\textbf{GW}(\Lambda)$, where each edge is colored in a single color and independently of other edges, and where color $i$ is attributed with probability $\lambda_i/\Lambda$ (see the proof of Proposition~\ref{prop:loc.conv} below).

Moreover, using that finite neighborhoods of two given points are mostly independent (in particular, they are unlikely to intersect), it is possible to strengthen the previous convergence in expectation into a convergence in probability, see e.g.~\cite[Theorem 2.19]{Hof16.2} in the case of uncolored graphs. In our case, we obtain the following result, which is our main tool for proving Theorem~\ref{thm:k>2}~\eqref{pt i} and Proposition~\ref{prop:Nell}. Although the proof is standard, we briefly sketch the argument for reader's convenience. A similar result is derived in the proof of Proposition~4.1.7 in~\cite{RVFM22.0}.

\begin{proposition}\label{prop:loc.conv}
Let $\varphi$ be a bounded function on the set of finite rooted graphs whose edges are endowed with a label encoding its set of colors. Then, for any $L\ge 1$, 
\begin{equation}\label{eq:local.conv}
\frac 1n \sum_{u\in [n]} \varphi(u,\mathcal V_L(u)) \xrightarrow[n\to \infty]{\mathbb P} \mathbb E[\varphi(\varnothing,\mathcal V_L(\varnothing))], 
\end{equation}
where $\mathcal V_L(\varnothing)$ is the $L$-neighborhood of the root $\varnothing$ in the colored graph $\emph{\textbf{GW}}(\Lambda)$ defined above. 
\end{proposition}
\begin{proof}
As already mentioned, the convergence in expectation is a straightforward consequence of the well-known local weak convergence of $G$. Note that conditionally on being present in $G$, an edge has color $i$ with probability
$$\frac{\lambda_i/n}{1-\prod_{j=1}^k(1-\lambda_j/n)}= (1+o(1))\frac{\lambda_i}{\Lambda}.$$
In particular, since $(\lambda_i/\Lambda)_{i=1}^k$ sum up to 1, a.s.\ every edge in the limit has only one color.

To prove convergence in probability, we use Chebyshev's inequality and bound the variance of the random variable appearing on the left-hand side of~\eqref{eq:local.conv}. Assume without loss of generality that $\varphi$ is non-negative and bounded by one. Then, on the one hand, for every $u \in [n]$ and every fixed $L\ge 1$, 
\begin{align}\label{eq:cov.1}
 \sum_{v\in [n]} \mathbb E[\varphi(u,\mathcal V_L(u))&  \varphi(v,\mathcal V_L(v))\cdot 
\mathds{1}_{\{ v\in \mathcal V_{2L}(u)\}} ]  \le \mathbb E[|\mathcal V_{2L}(u)|] \le 1+\dots+\Lambda^{2L}=\mathcal O(1). 
\end{align}
On the other hand, conditionally on the event that $v\notin \mathcal V_{2L}(u)$, or equivalently that $\mathcal V_L(u)\cap \mathcal V_L(v)=\emptyset$, 
and $|\mathcal V_L(u)|=m$ for some $m\ge 1$, $\mathcal V_L(v)$ is distributed as the $L$-neighborhood of $v$ in a colored ER random graph with $n-m$ vertices and parameter $1-\prod_{i=1}^k(1-\tfrac{\lambda_i}{n})$, denoted hereafter by $G_{n-m}'$. Thus, denoting also by $\mathbb E_{n-m}$ the expectation with respect to the distribution of $G_{n-m}'$, one has 
\begin{align*}
&   \mathbb E[\varphi(u,\mathcal V_L(u))  \varphi(v,\mathcal V_L(v))\cdot 
\mathds{1}_{\{ v\notin \mathcal V_{2L}(u)\}} ] \\  
  =
&   \sum_{m\ge 1} \mathbb E\left[\varphi(u,\mathcal V_L(u)) \cdot \mathbb E[\varphi(v,\mathcal V_L(v))\mathds 1_{\{\mathcal V_L(u)\cap \mathcal V_L(v)=\emptyset\}} \mid \mathcal V_L(u)]\cdot  \mathds 1_{\{|\mathcal V_L(u)|=m\}} \right] \\
  \le
&   \sum_{m\ge 1} \mathbb E\left[\varphi(u,\mathcal V_L(u)) \cdot  \mathds 1_{\{|\mathcal V_L(u)|=m\}} \right]\cdot \mathbb E_{n-m}[\varphi(v,\mathcal V_L(v))].
\end{align*}

Now, the set $\mathcal V_L(v)$ on $G_{n-m}'$ is the same as on $G$ unless there is at least one edge between one vertex of this set and one of the $m$ additional vertices which are present in $G$ but not in $G_{n-m}'$. Conditionally on the size of $\mathcal V_L(v)$ in $G_{n-m}'$, this holds with probability bounded from above by $m|\mathcal V_L(v)|\tfrac{\Lambda}{n}$. Therefore, 
$$\mathbb E_{n-m}[\varphi(v,\mathcal V_L(v))] \le 
\mathbb E_n[\varphi(v,\mathcal V_L(v))] + 2 m\mathbb E_{n-m}[|\mathcal V_L(v)|] \cdot \tfrac{\Lambda}{n} =\mathbb E_n[\varphi(v,\mathcal V_L(v))] + \mathcal O\bigg(\frac mn\bigg).$$ 
As a consequence, using that $\mathbb E[\varphi(u, \cV_L(u))\cdot |\mathcal V_L(u)|]\le \mathbb E[|\mathcal V_L(u)|]$, which is bounded uniformly in $n$, we get 
$$  \mathbb E[\varphi(u,\mathcal V_L(u))  \varphi(v,\mathcal V_L(v))\cdot 
\mathds{1}_{\{ v\notin \mathcal V_{2L}(u)\}} ] - \mathbb E[\varphi(u,\mathcal V_L(u))] \mathbb E[  \varphi(v,\mathcal V_L(v))] = \mathcal O\bigg(\frac 1n\bigg).$$ 
Together with~\eqref{eq:cov.1}, and summing over all pairs of vertices $u,v\in [n]$, this gives  
$$\text{Var}\left(\frac 1n \sum_{u\in [n]} \varphi(u,\mathcal V_L(u))\right) = \mathcal O\bigg(\frac 1n\bigg),$$
proving the desired concentration result. This concludes the (sketch of) proof of the proposition. 
\end{proof}

We can now give the proof of our main results. 

\begin{proof}[Proof of Theorem~\ref{thm:k>2}~\eqref{pt i}] 
Assume first that $\lambda_1^*\le 1$. 
In this case, it is well-known that the size of the largest connected component of $G^1$ divided by $n$ converges in probability to zero, and 
thus a fortiori the same must hold for the largest CA-component.

Assume now that $\lambda_1^* > 1$. In this case, a.a.s.\ each of $G^1, \ldots, G^k$ contains a unique giant component denoted by $\cC_{\max}^1, \ldots, \cC_{\max}^k$, respectively. When this is not the case for some $i\in [k]$, we define $\cC_{\max}^i$ to be an arbitrarily chosen largest component in $G^i$. Let also $\mu_i$ be the asymptotic proportion of vertices in $\cC_{\max}^i$. Since a.a.s.\ every non-giant connected component in $G^1, \ldots, G^k$ has size $\mathcal O(\log n)$ (which follows by combining Lemma~\ref{lem:sizeERC} with Theorem~4.15 in~\cite{Hof16}), it is sufficient to show that the size of $\bigcap_{i=1}^k \cC_{\max}^i$ divided by $n$ converges to a positive constant in probability. 

Firstly, note that by~\eqref{LLN.giant} one has
$\mu_i=  \mathbb P(|\cC^i(\varnothing)|=\infty)$, where we keep the notation $\cC^i(u)$ for the connected component of a vertex $u$ in $\textbf{GW}(\Lambda)$ after removal of all edges in color $i$. Thus, for every $i\in [k]$, 
$$\frac{|\cC^i_{\max}|}{n}= \frac 1n \sum_{u\in [n]} \mathds{1}_{\{u\in \cC^i_{\max}\}} \xrightarrow[n\to \infty]{\mathbb P}  \mathbb P(|\cC^i(\varnothing)|=\infty).$$
On the other hand, for every $L\ge 1$, by Proposition~\ref{prop:loc.conv} one has
\begin{equation*}
\frac 1n \sum_{u\in [n]} \mathds{1}_{\{|\cC^i(u)|\,\ge\, L\}} \xrightarrow[n\to \infty]{\mathbb P} \mathbb P(|\cC^i(\varnothing)|\ge L),
\end{equation*}
since the event of having a connected component of size at least $L$ is a measurable function of the $L$-neighborhood. Taking the difference between the terms in the last two displays, we deduce that for every $i\in [k]$ and every $L\ge 1$,
\begin{equation}\label{eq:largecc.notmax}
\frac 1n \sum_{u\in [n]} \mathds{1}_{\{|\cC^i(u)|\,\ge\, L \text{ and }  u\notin \cC^i_{\max} \} } \xrightarrow[n\to \infty]{\mathbb P} \mathbb P(L\le |\cC^i(\varnothing)|< \infty).
\end{equation}
Since the probability on the right-hand side  goes to $0$ as $L\to \infty$, for every $\varepsilon>0$ one can find $L$ such that for all $i\in [k]$, 
$$\limsup_{n\to\infty}\  \mathbb P\Big(\frac 1n \sum_{u\in [n]} \mathds{1}_{\{|\cC^i(u)|\,\ge\, L \text{ and }  u\notin \cC^i_{\max} \} } \ge \frac{\varepsilon}{k}\Big) = 0,$$
which by summation over $i$ gives 
\begin{equation}\label{eq:eps mass}
\limsup_{n\to\infty}\  \mathbb P\Big( \frac 1n \sum_{u\in [n]} \mathds{1}_{\{\exists i\in [k]: \, |\cC^i(u)|\,\ge\, L \text{ and }  u\notin \cC^i_{\max} \} } \ge \varepsilon\Big) = 0.
\end{equation}
Moreover, using Proposition~\ref{prop:loc.conv} again yields 
\begin{equation*}
\frac 1n \sum_{u\in [n]} \mathds{1}_{\{|\cC^i(u)|\,\ge\, L \text{ for all } i\in [k]\}}\xrightarrow[n\to \infty]{\mathbb P} \mathbb P(|\cC^i(\varnothing)|\ge L \text{ for all }i\in [k]).    
\end{equation*}
Then, letting $L\to \infty$ together with~\eqref{eq:eps mass} implies that
$$\frac{|\bigcap_{i=1}^k \cC^i_{\max}|}{n} =\frac 1n \sum_{u\in [n]} \mathds{1}_{\{u \in \cC^i_{\max} \text{ for all }i\in [k]\} } \xrightarrow[n\to \infty]{\mathbb P} \mathbb P(|\cC^i(\varnothing)|=\infty \text{ for all }i\in [k]).$$

To conclude the proof of Theorem~\ref{thm:k>2}~\eqref{pt i}, we show that the above limit is positive. For every $L\ge 1$ and a fixed vertex $u\in [n]$, the events $\{|\cC^i(u)|\ge L\}$ are increasing, so by the FKG inequality (see Theorem 3.1 in~\cite{Gri06})
\[\mathbb P(|\cC^i(u)|\ge L \text{ for all }i\in [k])\ge \prod_{i=1}^{k} \Pr(|\cC^i(u)|\ge L).\] 
Then, letting $n\to \infty$ and using the local convergence of $G$ towards $\mathbf{GW}(\Lambda)$ implies that
\[\mathbb P(|\cC^i(\varnothing)|\ge L \text{ for all }i\in [k])\ge \prod_{i=1}^{k} \Pr(|\cC^i(\varnothing)|\ge L).\]
Finally, letting $L\to \infty$ shows that
\[\mathbb P(|\cC^i(\varnothing)|=\infty \text{ for all }i\in [k])\ge \prod_{i=1}^{k} \Pr(|\cC^i(\varnothing)| = \infty) > 0,\]
as desired.
\end{proof}

\begin{proof}[Proof of Proposition~\ref{prop:Nell}] 
Firstly, we recall the notion of CA-connectivity in $\textbf{GW}(\Lambda)$ from~\cite{RVFM22.0}. Two vertices $u$ and $v$ are declared to be CA-connected if for every $i\in [k]$, either $u$ and $v$ are connected in the subgraph of $\textbf{GW}(\Lambda)$ obtained by removing edges in color $i$, which we denote by  $\textbf{GW}^i(\Lambda)$, or if their connected components in this graph are both infinite. We denote by $\widetilde \cC(\varnothing)$ the CA-component of the root. 

Now, we define $I=\{i:\lambda_i^*>1\}$, and $J=[k]\setminus I$. Also, given two vertices $u$ and $v$ in 
%some vertices of 
$\textbf{GW}(\Lambda)$ and a subgraph $H$ of $\textbf{GW}(\Lambda)$, we denote by $u\xleftrightarrow{H} v$ the event that $u$ and $v$ are connected in $H$. Then, for $L\ge 1$, $M\ge 1$, and $u\in [n]$, we define
$$\widetilde \cC_{L,M}(u)=\Big(\bigcap_{i\in J} \Big\{v: u\xleftrightarrow{G^i\cap \cV_L(u)} v\Big\}\Big)\cap \Big(\bigcap_{i\in I} \ \Big\{v\in \mathcal V_L(u):  u\xleftrightarrow{G^i\cap \cV_L(u)} v \ \text{or}\  \min(|\cC^i(u)|,|\cC^i(v)|)\ge M \Big\}\Big),$$
with the notation from the proof of Theorem~\ref{thm:k>2}~\eqref{pt i}, and 
$$\widetilde \cC_L(u)=
\Big(\bigcap_{i\in J} \Big\{v\in \mathcal V_L(u): u\xleftrightarrow{G^i\cap \cV_L(u)} v\Big\}\Big)\cap \Big(\bigcap_{i\in I} \ \Big\{v\in \mathcal V_L(u):  u\xleftrightarrow{G^i\cap \cV_L(u)} v \ \text{or both}\ u,v\in \cC^i_{\max}  \Big\}\Big).$$
Note that on the a.a.s.\ event $|\cC_{\max}^i| \ge \tfrac{\mu_i}{2} n$ and for all sufficiently large $n$, every vertex $u$ such that $|\widetilde \cC_{L}(u)| < |\widetilde \cC_{L,M}(u)|$ is at distance at most $L$ from the set 
\[\cS_M = \bigcup_{i\in I} \{v\in G: |\cC^i(v)|\ge M \text{ and }v\notin \cC^i_{\max}\}.\]
However, by the Cauchy-Schwarz inequality, the expected number of such sites divided by $n$ is at most
$$\mathbb E\left[\frac 1n\sum_{v\in [n]} |\cV_{L}(v)|\mathds 1_{\{v\in \cS_M\}}\right] = 
\mathbb E\left[|\cV_{L}(1)|\mathds 1_{\{1\in \cS_M\}}\right] \le \mathbb E[|\cV_{L}(1)|^2]^{1/2} \cdot \mathbb P(1\in \cS_M)^{1/2}.
$$ 
Moreover, by~\eqref{eq:largecc.notmax}, $\mathbb P(1\in \cS_M)\to 0$ as $M\to \infty$ uniformly in $n$, and a straightforward computation shows that for fixed $\Lambda$ and $L$, the second moment of $|\cV_{L}(1)|$ is uniformly bounded in $n$ (using e.g. that it is stochastically dominated by the size of the $L$-neighborhood of the root in a BGW tree with $\text{Bin}(n,\tfrac{\Lambda}n)$ offspring distribution).  Therefore, Markov's inequality implies that for every $\varepsilon>0$ and every $L$, one can find $M$ large enough so that
\begin{equation*}
\limsup_{n\to \infty} \ \mathbb P\left(\frac 1n\sum_{v\in [n]} |\cV_{L}(v)|\mathds 1_{\{v\in \cS_M\}}\ge \varepsilon\right)\le \eps,
\end{equation*}
and in particular,
\begin{equation}\label{eq:conv 1}
\limsup_{n\to \infty} \ \mathbb P\left(\frac{|\{u\in G: |\widetilde \cC_{L}(u)|<|\widetilde \cC_{L,M}(u)|\}|}{n}\ge \varepsilon\right) \le \limsup_{n\to \infty} \Pr\left(\exists i\in [k]: |\cC_{\max}^i|\le \frac{\mu_i n}{2} \right) + \eps = \eps.   
\end{equation}
Similarly, we define 
$$\widetilde \cC_{L,M}(\varnothing)=
\Big(\bigcap_{i\in J} \Big\{v\in \mathcal V_L(\varnothing)\cap \cC^i(\varnothing)\Big\}\Big)\cap \Big(\bigcap_{i\in I} \ \Big\{v\in \mathcal V_L(\varnothing):  v\in \cC^i(\varnothing)\ \text{or}\  \min(|\cC^i(\varnothing)|,|\cC^i(v)|)\ge M \Big\}\Big),$$
and 
$$\widetilde \cC_L(\varnothing)=
\Big(\bigcap_{i\in J} \Big\{v\in \mathcal V_L(\varnothing)\cap \cC^i(\varnothing)\Big\}\Big)\cap \Big(\bigcap_{i\in I} \ \Big\{v\in \mathcal V_L(\varnothing):  v\in \cC^i(\varnothing)\ \text{or}\  |\cC^i(\varnothing)|=|\cC^i(v)|=\infty \Big\}\Big),$$
which is the decreasing limit of $\widetilde \cC_{L,M}(\varnothing)$ as $M\to \infty$. 
Moreover, 
\begin{align*}
    \Pr\Big(|\widetilde \cC_{L}(\varnothing)|<|\widetilde \cC_{L,M}(\varnothing)|\Big) \le  \mathbb E\Big[|\widetilde \cC_{L,M}(\varnothing)| - |\widetilde \cC_L(\varnothing)|\Big] \le \sum_{i\in I} \mathbb E\left[\sum_{v\in \mathcal V_L(\varnothing)} \mathds 1_{\{v\notin \cC^i(\varnothing),\, M\,\le\, |\cC^i(v)|\, <\, \infty\}}\right]. 
\end{align*}
Fix $i\in I$. For every edge $e\in \textbf{GW}(\Lambda)$, let us denote by $e^+$ the endvertex of $e$ which is farther from the root. Then, for every vertex $v\in \textbf{GW}(\Lambda)$ such that $\cC^i(\varnothing) \neq \cC^i(v)$, the (unique) path between $\varnothing$ and $v$ in $\textbf{GW}(\Lambda)$ contains an edge in color $i$. Considering the closest such edge to $v$, we get
$$\sum_{v\in \mathcal V_L(\varnothing) }\mathds{1}_{\{v\notin \cC^i(\varnothing),\, M\,\le\, |\cC^i(v)|\, <\, \infty\}} 
\le \sum_{e\in \mathcal V_L(\varnothing),\, \text{$e$ in color $i$}} |\cC^i(e^+)\cap \mathcal V_L(e^+)|\cdot \mathds 1_{\{M\,\le\, |\cC^i(e^+)|\,<\,\infty\}}.  
$$
Since for any edge $e$ in color $i$, the component $\cC^i(e^+)$ is contained in the subtree of the descendants of $e^+$, and is thus independent of the remainder of $\textbf{GW}(\Lambda)$, we have 
$$\mathbb E\left[\sum_{v\in \mathcal V_L(\varnothing) }\mathds{1}_{\{v\notin \cC^i(\varnothing),\, M\,\le\, |\cC^i(v)|\, <\, \infty\}} \right]
\le \mathbb E[|\mathcal V_L(\varnothing)|] \cdot \mathbb E\left[|\mathcal V_L(\varnothing)|\cdot \mathds 1_{\{M\,\le\, |\cC^i(\varnothing)|\,<\,\infty\}}\right]. 
$$
Then, using Cauchy-Schwarz inequality as before, we get that
\begin{equation}\label{eq:conv 2}
\Pr(|\widetilde \cC_{L}(\varnothing)|<|\widetilde \cC_{L,M}(\varnothing)|)\xrightarrow[M\to \infty]{} 0. 
\end{equation}
The next step is to notice that since the sets $\widetilde \cC_{L,M}(u)$ are measurable with respect to the $(L+M)$-neighborhood of a vertex $u$, for every $\ell\ge 1$, Proposition~\ref{prop:loc.conv} implies that 
$$\frac{1}{n}\sum_{u\in [n]}\mathds{1}_{\{|\widetilde \cC_{L,M}(u)|\,=\,\ell\}} \xrightarrow[n\to \infty]{\mathbb P} \mathbb P(|\widetilde \cC_{L,M}(\varnothing)|= \ell).$$
We remark that this last step is reminiscent of a similar convergence in Lemma~5.2.4 in~\cite{RVFM22.0}. Together with~\eqref{eq:conv 1} and~\eqref{eq:conv 2}, by letting $M\to \infty$ we get that for any $L\ge 1$,
\begin{equation}\label{eq:conv.fixedL}
\frac{1}{n}\sum_{u\in [n]}\mathds{1}_{\{|\widetilde \cC_L(u)|\,=\,\ell\}} \xrightarrow[n\to \infty]{\mathbb P} \mathbb P(|\widetilde \cC_L(\varnothing)|= \ell).
\end{equation}
Finally, to conclude the proof of~\eqref{eq:asympt.CA}, it amounts to consider the $L\to \infty$ limit in the last display. For the right-hand side, we just observe that the CA-component of the root is the increasing limit of the sets $\widetilde \cC_L(\varnothing)$ as $L\to \infty$, from which it follows that for any $\ell \ge 1$, 
$$ 
\mathbb P(|\widetilde \cC_L(\varnothing)|=\ell, |\widetilde \cC(\varnothing)|>\ell) \xrightarrow[L\to \infty]{} 0,$$
which yields for any $\ell\ge 1$, 
\begin{equation}\label{eq:conv.Linfty1}
\mathbb P(|\widetilde \cC_L(\varnothing)|= \ell) \xrightarrow[L\to \infty]{} \mathbb P(|\widetilde \cC(\varnothing)|=\ell). 
\end{equation}
It remains to show the corresponding convergence for a typical vertex of $G$. We distinguish two cases.  
If the set $J$ is nonempty (or equivalently if $\lambda_1^*\le 1$), then we claim that for any vertex $u\in [n]$, we have
\begin{equation}\label{inclusion.Jnonempty}
\{\widetilde \cC_L(u) \neq \widetilde \cC(u)\}
\ \subseteq \  \left(\bigcup_{i\in J} \{|\cC^i(u)|\ge L\}\right) \cup \left(\bigcup_{i\in I} \{|\cC^i(u)|\ge L \text{ and } u\notin \cC_{\max}^i\}\right).
\end{equation}
Indeed, for the event on the left-hand side of~\eqref{inclusion.Jnonempty} to hold, 
either there is a vertex in $\widetilde \cC(u)\setminus \mathcal V_L(u)$, which together with $\widetilde \cC(u)\subseteq \cC^1(u)$ implies that $|\cC^1(u)|\ge L$, or there is a vertex $v$ in $\widetilde \cC(u)\cap \mathcal V_L(u)$ outside $\widetilde \cC_L(u)$, which means that there is $i\in [k]$ such that $u$ and $v$ are connected by a path in $G^i$ exiting $\mathcal V_L(u)$, and if $i\in I$, additionally $\cC^i(u)\neq \cC^i_{\max}$.
Now, given $\varepsilon > 0$, one can choose $L$ large enough so that $\mathbb P(|\cC^i(\varnothing)|\ge L)\le \tfrac{\varepsilon}{2k}$ for all $i\in J$ and $\mathbb P(L\le |\cC^i(\varnothing)| < \infty)\le \tfrac{\varepsilon}{2k}$ for all $i\in I$, and then using Proposition~\ref{prop:loc.conv}, we get that
\begin{equation}
\begin{split}\label{eq:caseJnonempty}
&\limsup_{n\to \infty}\ \mathbb P\left(\frac{|\{u\in G: |\widetilde \cC_{L}(u)|<|\widetilde \cC(u)|\}|}{n}
\ge \varepsilon \right)\\
\le\;
&\sum_{i\in J}\limsup_{n\to \infty}\ \mathbb P\left(\tfrac{|\{u\in G\,:\, |\cC^i(u)|\,\ge\, L\}|}{n}\ge \frac{\varepsilon}{k} \right) + \sum_{i\in I}\limsup_{n\to \infty}\ \mathbb P\left(\tfrac{|\{u\in G\,:\, L\,\le\, |\cC^i(u)|\, <\, \infty\}|}{n}\ge \frac{\varepsilon}{k} \right) = 0.
\end{split} 
\end{equation}
We consider now the slightly more difficult case when 
$J$ is empty. Let $A= \bigcap_{i=1}^k \cC^i_{\max}$. One has 
\begin{equation}\label{eq:conv.Linfty2}
\big\{|\widetilde \cC_L(u)|= \ell, |\widetilde \cC(u)|>\ell \big\}\ \subseteq \ \Big\{|\widetilde \cC_L(u)|= \ell, u\in A\Big\} \cup\left(\bigcup_{i=1}^k  \{|\cC^i(u)|\ge L, \cC^i(u)\neq \cC^i_{\max}\}\right) 
\end{equation}
since either $u\in A$, or there is a vertex $v$ and $i\in [k]$ such that $u,v\notin \cC^i_{\max}$ but $u$ and $v$ are connected by a path in $G^i$ exiting $\cV_L(u)$. Moreover, note that if $u\in A$, then $\widetilde \cC(u)= A$, and thus 
$$\Big\{|\widetilde \cC_L(u)|= \ell, u\in A\Big\}\ \subseteq\ \Big\{\big|A\cap \mathcal V_L(u)\big| =\ell, |\mathcal V_L(u)|\ge L \Big\}\cup \{|A| = \ell\}.$$
Indeed, if $u\in A$ and $|A|\ge \ell+1$, then $A$ must necessarily contain a vertex outside $\cV_L(u)$, which means that $|\cV_L(u)|\ge L$. Note that Theorem~\ref{thm:k>2}~\eqref{pt i} implies that $\Pr(|A|=\ell)\to 0$ as $n\to \infty$, so we concentrate on the first event in the union above. For $M\ge 1$, define
$$A_M= \{v : |\cC^i (v)|\ge M \text{ for all }i\in [k]\},$$
where for convenience we see $A_M$ as a vertex subset of both $G$ or  $\textbf{GW}(\Lambda)$.  
By a similar argument as for~\eqref{eq:conv 1}, we know that for every $\varepsilon>0$, there exists $M$ such that
$$\limsup_{n\to \infty} \ \mathbb P\left(\Big|\frac 1n \sum_{u\in [n]} \mathds 1_{\{|A\cap \mathcal V_L(u)|\,=\,\ell,\,|\mathcal V_L(u)|\,\ge\, L\}} - \frac 1n \sum_{u\in [n]} \mathds 1_{\{|A_M\cap \mathcal V_L(u)|\,=\,\ell,\,|\mathcal V_L(u)|\,\ge\, L\}}\Big|\ge \varepsilon\right)=0. $$
However, by Proposition~\ref{prop:loc.conv} one has for any $L\ge 1$ and $M\ge 1$, 
$$\frac 1n \sum_{u\in [n]} \mathds 1_{\{|A_M\cap \mathcal V_L(u)|\,=\,\ell,\,|\mathcal V_L(u)|\,\ge\, L\}}\xrightarrow[n\to\infty]{\mathbb P} \mathbb P\big(|A_M\cap \mathcal V_L(\varnothing)|=\ell,|\mathcal V_L(\varnothing)|\ge L\big), $$ 
which goes to $0$ as $L\to \infty$ for any fixed $M\ge 1 $. Indeed, this can be seen by exploring $\mathcal V_L(u)$ in several steps.  
At each step, we fix a vertex $v$ at the boundary of the already explored set and explore the successors of $v$ at distance at most $M$ from $v$. Note that the probability that 
$v$ belongs to $A_M$ is bounded from below by a positive constant, which is  independent of the previous steps. Furthermore, as $L\to \infty$ and on the event $\{|\mathcal V_L(\varnothing)|\ge L\}$, the number of steps goes almost surely to infinity, and thus the probability of the event $\{|A_M\cap \mathcal V_L(\varnothing)|=\ell\}$ tends to 0 for every fixed $\ell$.  
Then, by using also~\eqref{eq:eps mass} again to handle the second union in~\eqref{eq:conv.Linfty2}, we deduce that for every fixed $\eps > 0$ and $\ell\ge 1$,  there is a sufficiently large $L$ so that,
$$\limsup_{n\to\infty} \Pr\left(\frac{|\{u\in G: |\widetilde \cC_L(u)| = \ell, |\widetilde \cC(u)| > \ell\}|}{n}\ge \eps\right) = 0.$$ Together with \eqref{eq:conv.fixedL}, \eqref{eq:conv.Linfty1} and~\eqref{eq:caseJnonempty}, this proves~\eqref{eq:asympt.CA}.

The last piece of the proof of the proposition is to show that for any $\ell \ge 2$, $\nu_\ell = \mathbb P(|\widetilde \cC(\varnothing)|=\ell)$ is positive if and only if $\lambda_k^*>1$, and that the constant $a_1$ appearing in Part~\eqref{pt i} of Theorem~\ref{thm:k>2} is equal to the probability that the CA-component of the root is infinite. This part bears close resemblance to the proof of Proposition 2.18~(ii) from~\cite{RVFM22.0}. For the first part, assume that $\lambda_k^* > 1$, and let  $\ell \ge 1$ be given. Then, with positive probability one can have altogether $|\cC_k(\varnothing)|=\ell$, all vertices of $\cC_k(\varnothing)$ are connected to infinity in $\textbf{GW}^k(\Lambda)$, and  $|\cC^i(\varnothing)|<\infty$ for all values of $i$ different from $k$. If the last events hold simultaneously, one can observe that the CA-component of the root is precisely $\cC_k(\varnothing)$, and thus it has size $\ell$. Conversely, if $\lambda_k^*\le 1$, then it is well-known that all the components $\cC^i(\varnothing)$ are a.s. finite, and since edges have a.s.\ a unique color, this implies that the CA-component of the root is necessarily reduced to a single vertex.

For the last part, note that on the one hand, if $|\widetilde \cC(\varnothing)| = \infty$, then $|\cC^i(\varnothing)| = \infty$ for all $i\in [k]$, so
\[a_1 = \mathbb P(|\cC^i(\varnothing)| = \infty \text{ for all $i\in [k]$})\ge \mathbb P(|\widetilde \cC(\varnothing)| = \infty).\]
The proof of the reverse inequality appears as Lemma~5.1.1 in~\cite{RVFM22.0}; we provide it for completeness. First, notice that if $a_1=0$, there is nothing to prove. Thus, we may assume that $a_1>0$, which by Part~\eqref{pt i} of Theorem~\ref{thm:k>2} is equivalent to $\lambda_1^*>1$. Suppose that each of the components $\cC^i(\varnothing)$ for $i\in [k]$ is infinite. We prove that in this case, the CA-component of the root is also a.s.\ infinite. To begin with, the well-known Kesten-Stigum theorem~\cite{KS66} implies that on the event $\{|\cC^1(\varnothing)| = \infty\}$, the number of vertices in generation $L$ in the tree $\cC^1(\varnothing)$ goes to infinity almost surely as $L\to \infty$. However, conditionally on the vertices $v_1,\dots,v_N$ in the $L$-th generation of that tree, 
the subtrees $T_1,\dots,T_N$ of $\textbf{GW}(\Lambda)$ emanating from the vertices $v_1,\dots,v_N$, respectively, are all independent. Thus, the probability that for at least $a_1N/2$ of them, the connected components of their root are infinite in each of $\textbf{GW}^i(\Lambda)$ for $i\in [k]$, goes to one as $N\to \infty$. Moreover, all such vertices are in the CA-component of $\varnothing$. But as already mentioned, on the event $\{|\cC^1(\varnothing)|= \infty\}$ we have that $N=N(L)\to\infty$ almost surely as $L\to \infty$, so the result follows.
\end{proof}

\subsection{\texorpdfstring{Proof of Theorem~\ref{thm:k>2}~\eqref{pt ii} and Proposition~\ref{prop:critical}}{}} \label{sec:proof.ii}

The proof starts with the following general lemma that will also be used for the proof of Theorem~\ref{thm:k>2}~\eqref{pt iii} in the next section. For $I\subset [k]$, write $\lambda_I^* = \Lambda - \sum_{i\in I} \lambda_i$.

\begin{lemma}\label{lem:Ijconnect}
For every $j\in [k]$ and $I\subseteq [k]\setminus \{j\}$ such that $\max(\lambda_I^*,\lambda_j^*)<1$, one has uniformly in $n$, 
$$\lim_{M\to \infty} \mathbb P\big(\exists u\neq v: u\xleftrightarrow{G^I}v,\,  u\xleftrightarrow{G^j}v,\, u\stackrel{G^{I\cup\{j\}}}{\centernot\longleftrightarrow}v, \, \max(|\cC^I(u)|,|\cC^j(u)|)\ge M\big) = 0.$$
\end{lemma}
\begin{proof}
Fix $\varepsilon>0$ and a set $I\subseteq [k]\setminus \{j\}$ with $\max(\lambda_I^*,\lambda_j^*)<1$. For every $u,v\in [n]$, define the event
$$\cA_{u,v}:=\big\{u\xleftrightarrow{G^I}v, \, u\xleftrightarrow{G^j}v, \, u\stackrel{G^{I\cup\{j\}}}{\centernot\longleftrightarrow} v\big\}.$$
Then, on the event $\cA_{u,v}$, 
there exists a path $P$ connecting $u$ and $v$ not using color $j$. Since $u$ and $v$ are not connected in $G^{I\cup \{j\}}$, $P$ must contain at least one edge which is not in $G^I$. 
In particular, since $v$ is a vertex in $\cC^I(u)$, $P$ must exit the graph $\cC^I(u)$ at some point and return to it later.
%at some other point}. 
Consequently, on the event $\cA_{u,v}$, there must exist 
$w_1,w_2 \in \cC^I(u)$ which are connected in $G^j\setminus \cC^I(u)$ (where the latter graph is the subgraph of $G^j$ with vertex set $[n]$ obtained by removing only the edges of $\cC^I(u)$). Moreover, note that conditioning on $\cC^I(u)$ only forces edges in $G\setminus \cC^I(u)$ with at least one endvertex in $\cC^I(u)$ to have a color in $I$, while it does not reveal any information for the remaining edges.  
Thus, for every fixed graph $H$ containing $u$, one has
$$\mathbb P\Big(w_1\xleftrightarrow{G^j\setminus H}w_2
\mid \cC^I(u)=H\Big) \le \mathbb P\Big(w_1\xleftrightarrow{G^j\setminus H}w_2\Big) \le \mathbb P(w_1\xleftrightarrow{G^j}w_2) = \mathcal O\left(\frac 1n\right),$$
where for the first inequality we use that conditioning on $\{\cC^I(u)=H\}$ makes the event $\{w_1\xleftrightarrow{G^j\setminus H}w_2\}$ harder since, by the previous discussion, it forces edges emanating from $w_1$ and $w_2$ to have a color in $I$, 
%by the discussion above}, 
and for the last equality we use Corollary~\ref{cor:connectionER} together with the fact that $G^j$ is stochastically dominated by $G(n,\lambda_j^*/n)$ (and that $\lambda_j^*<1$ by hypothesis). As a consequence, on the event $\{v\in \cC^I(u)\}$,
$$\mathbb P\Big(u\xleftrightarrow{G^j}v, \, u\stackrel{G^{I\cup\{j\}}}{\centernot\longleftrightarrow} v
\mid \cC^I(u)\Big)\le \sum_{w_1,w_2\in \cC^I(u)} 
\mathbb P\Big(w_1\xleftrightarrow{G^j\setminus \cC^I(u)}w_2
\mid \cC^I(u)\Big) = \cO\left(\frac{|\cC^I(u)|^2}{n}\right).$$ 
Therefore, 
\begin{align}
\mathbb P\big(\cA_{u,v}, \, |\cC^I(u)|\ge M\big)
& = \cO\left(\frac{1}{n}\sum_{t\,\ge\, M} t^2 \cdot \mathbb P(v\in \cC^I(u)\mid |\cC^I(u)|=t) \cdot \mathbb P(|\cC^I(u)| = t)\right)\nonumber \\
& = \cO\left(\frac{1}{n^2} \sum_{t\,\ge\, M} t^3 \cdot \mathbb P(|\cC^I(u)|=t)\right) = \cO\left(\frac {M^3\, \e^{-M\cdot I_{\lambda_I}}}{n^2}\right),\label{eq:part 1}
\end{align}
where we use~\eqref{prob.cond.connect} for the second equality and Lemma~\ref{lem:sizeERC} for the last one. Also, by a similar argument
\begin{equation}\label{eq:part 2}
\mathbb P\big(\cA_{u,v}, \, |\cC^j(u)|\ge M\big)=\cO\left(\frac {M^3\, \e^{-M\cdot I_{\lambda_j}}}{n^2}\right).
\end{equation} 
A union bound by summing~\eqref{eq:part 1} and~\eqref{eq:part 2} over all possible pairs of vertices $\{u,v\}$ of $G$, and letting $M\to \infty$ finishes the proof.
\end{proof}

We call a set \emph{connected} if its vertices belong to the same connected component.

\begin{lemma}\label{lem:intersection}
Suppose that $\lambda_{k-1}^* < 1$. Then, uniformly in $n$, 
\[\lim_{M\to \infty} \Pr(\exists S\subseteq [n]: |S|\ge M, \text{$S$ is connected in each of $G^1,\dots,G^{k-1}$ but not in $G_k$}) = 0.\]
\end{lemma}
\begin{proof}
Fix $S\subseteq [n]$ of size at least $M$, and assume that $S$ is connected in each of the graphs $G^1,\dots,G^{k-1}$. If $S$ is not connected in $G_k$, it means that one can find two distinct vertices $u,v\in S$ which are not connected in $G_k$. Let $I$ be a maximal subset of $[k-1]$ such that $u$ and $v$ are connected in $G^I$. Since $u$ and $v$ are not connected in $G_k$, $I\neq [k-1]$, and hence there exists $j\in [k-1]\setminus I$ such that $u$ and $v$ are not connected in $G^{I\cup \{j\}}$. However, using that $\max(\lambda_I^*, \lambda_j^*) < 1$ and $|\cC^j(u)|\ge |S|\ge M$, Lemma~\ref{lem:Ijconnect} ensures that this event happens with probability converging to 0 as $M\to \infty$, uniformly in $n$.
\end{proof}

We are ready conclude the proofs of Theorem~\ref{thm:k>2}~\eqref{pt ii} and Proposition~\ref{prop:critical}.

\begin{proof}[Proof of Proposition~\ref{prop:critical}]
Denote by $\widetilde \cC_{\max}$ the largest CA-component in $G$. As the vertices of $\widetilde \cC_{\max}$ form a connected set in $G^i$ for every $i\in [k-1]$, by Lemma~\ref{lem:intersection} we have that uniformly in $n$,
\begin{equation}\label{eq.connected.Ck}
\lim_{M\to \infty}\ \Pr(|\widetilde \cC_{\max}|\ge M, \text{$\widetilde \cC_{\max}$ is not connected in $G_k$}) = 0.
\end{equation}
On the other hand, if $\widetilde \cC_{\max}$ is connected in $G_k$, it is obtained by intersecting a connected component in $G^k$ and one in $G_k$. Moreover, by definition the probability of having an edge present in $G^k$ is $1-\prod_{i=1}^{k-1}(1-\frac{\lambda_i}{n})\le \tfrac{\lambda_k^*}{n}=\tfrac{1}{n}$. Thus, $G^k$ is stochastically dominated by $G(n,1/n)$. However, it is well-known that the size of the largest component in $G(n, 1/n)$ is a.a.s. of order $n^{2/3}$ (see Proposition 5.2 in~\cite{Hof16}), in particular it is a.a.s.\ smaller than $n^{3/4}$, say. On the other hand, by a similar argument as in~\eqref{prob.cond.connect} one has that for any $n > \ell\ge M\ge 2$ and any distinct vertices $v_1,\dots,v_M\in [n]$, 
$$\mathbb P\big(v_M\in \cC^k(v_1)\mid |\cC^k(v_1)|=\ell, v_2,\dots,v_{M-1}\in \cC^k(v_1)\big) = \frac{\ell - (M-1)}{n-(M-1)}\le \frac{\ell}{n},$$ 
and the same holds with $\cC_k(v_1)$ instead of $\cC^k(v_1)$. Thus, for any $M\ge 2$ and any distinct $v_1,\dots,v_M\in [n]$, 
\begin{align*}
 \mathbb P\big(v_1,\dots,v_M \text{ are connected in }G^k, |\cC^k_{\max}|\le n^{3/4}\big)
& \le \sum_{\ell = M}^{n^{3/4}} \bigg(\frac{\ell}{n}\bigg)^{M-1} \cdot \mathbb P(|\cC^k(v_1)|=\ell) \\
& =\frac{\mathbb E[|\cC^k(v_1)|^{M-1}\cdot \mathds 1_{\{|\cC^k(v_1)|\le n^{3/4}\}}]}{n^{M-1}}\\
& \le n^{-\frac{M-1}{4}}.  
\end{align*}
Likewise, we know by~\eqref{LLN.subcritical} that a.a.s.\ the size of the largest connected component in $G_k$ is at most $\tfrac{2}{I_{\lambda_k}}\log n$, and 
as above one has
$$\mathbb P\big(v_1,\dots,v_M \text{ are connected in }G_k, |\cC_k(v_1)|\le \tfrac 2{I_{\lambda_k}}\log n\big) = \mathcal O\bigg(\frac{(\log n)^{M-1}}{n^{M-1}}\bigg).$$
Summing over all possible vertices $v_1,\dots,v_M\in [n]$ and using independence between $G^k$ and $G_k$, we deduce that 
\begin{align*}
&\mathbb P(\exists v_1,\dots,v_M\text{ connected in both }G_k \text{ and }G^k) \\
& \le \mathbb P\bigg(\max_{u\in [n]} |\cC^k(u)|\ge n^{3/4}\bigg) 
+ \mathbb P\bigg(\max_{u\in [n]} |\cC_k(u)|\ge \tfrac 2{I_{\lambda_k}} \log n\bigg) + \mathcal O\bigg(\frac{(\log n)^{M-1}}{n^{\frac{M-5}{4}}}\bigg)=o(1), 
\end{align*}
where the last equality holds as soon as $M\ge 6$. 
Together with~\eqref{eq.connected.Ck}, this concludes the proof of the proposition.
\end{proof}

\begin{proof}[Proof of Theorem~\ref{thm:k>2}~\eqref{pt ii}]
Recall that now $\lambda_k^*>1>\lambda_{k-1}^*$, and in particular, all graphs $G^i$ with $i \le k-1$ are subcritical while $G^k$ is supercritical. 
Firstly, we observe that there exists $\varepsilon>0$ such that a.a.s.\ the largest CA-component has size larger than $\varepsilon \log n$. Indeed, we know by~\eqref{LLN.subcritical} that a.a.s.\ there exists a connected component in $G_k$ of size at least $\log n/(2I_{\lambda_k})$, and hence~\eqref{LLN.giant} and Lemma~\ref{lem:stoch.giant} imply together that a.a.s.\ its intersection with the giant component of $G^k$ has size at least $\mu_{\lambda_k^*} \log n/(4I_{\lambda_k})$ (recall that $G^k$ and $G_k$ are independent). 

Next, let $\widetilde \cC_{\max}$ be the largest CA-component. By definition its vertices are connected in all the graphs $G^1,\dots,G^{k-1}$, and thus by Lemma~\ref{lem:intersection} a.a.s.\ they are also connected in $G_k$. This means that $\widetilde \cC_{\max}$ is in fact obtained as the intersection of a connected component of $G_k$ with one of $G^k$. However, it is well-known that a.a.s.\ all connected components in a supercritical ER random graph but the largest one have size $\cO(\log n)$  (see e.g. Section 4.4.1 in~\cite{Hof16}). Thus, by the same argument as in the proof of Proposition~\ref{prop:critical}, we deduce that the probability of having three vertices connected in $G_k$ and participating in the same non-giant component of $G^k$ is $\mathcal O(n^3 \cdot \frac{(\log n)^4}{n^4}) = o(1)$. Hence, a.a.s.\ every CA-component of size at least 3 (and $\widetilde \cC_{\max}$ in particular) is contained in the giant in $G^k$.

Finally, by Lemma~\ref{lem:stoch.giant} and Corollary~\ref{cor:LLN.inter.giant} we conclude that, with the notation of Corollary~\ref{cor:LLN.inter.giant},
$$\frac{|\widetilde \cC_{\max}|}{\log n} \xrightarrow[n\to \infty]{\mathbb P} a(\mu_{\lambda_k^*},\lambda_k),$$
which finishes the proof.
\end{proof}

\subsection{\texorpdfstring{Proof of Theorem~\ref{thm:k>2}~\eqref{pt iii}}{}}\label{sec:proof.iii}

We assume throughout this section that $\lambda_k^*<1$. 

We call \emph{support} of a CA-component the subgraph of $G$ obtained as the union of all paths in $G^1, \ldots, G^k$ between any pair of distinct vertices of the CA-component. The main observation of the proof is the following lemma.

\begin{lemma}\label{cor:support.cycle}
A.a.s.\ every CA-component is supported by either a single vertex, a single edge or a cycle of $G$. In particular, a.a.s.\ every CA-component has size at most $k$. 
\end{lemma}

\begin{proof}
Consider a CA-component $\widetilde \cC$ and assume that it is not reduced to a single vertex. Let $u$ and $v$ be two different vertices of $\widetilde \cC$. Assume first that $|\cC^i(u)|\ge \frac{c_1(\Lambda)}{k} \log n$ for some $i\in [k]$, with the notation of Remark~\ref{rem:2 cycles}. By the second point of this remark, and since $\lambda_i^*<1$ by hypothesis, we know that a.a.s.\ $\cC^i(u)$ is a tree with no repeated edge. In other words, $u$ and $v$ are connected by a unique path $P$ in $G^i$, and since $P$ contains no repeated edges, $u$ and $v$ cannot be connected in $G^{\{i,j\}}$ for any color $j$ in $P$. However, Lemma~\ref{lem:Ijconnect} applied for $I = \{i\}$ shows that a.a.s.\ this situation does not happen. Thus, we may assume that $|\cC^i(u)|\le \frac{c_1(\Lambda)}{k} \log n$ for all $i$, and by summation over $i$ we may as well assume that $\cC(u)$, the connected component of $u$ in $G$, has size at most $c_1(\Lambda)\log n$. Then, using the first result from Remark~\ref{rem:2 cycles}, we know that a.a.s. either $\cC(u)$ contains no cycles and at most one repeated edge or no repeated edges and at most one cycle. Moreover, note that for every pair of vertices $u$ and $v$ in $\widetilde \cC$, $u$ and $v$ cannot be disconnected in $G$ by deleting an edge with a single color in $\cC(u)$. Thus, the unique cycle or repeated edge necessarily supports $\widetilde \cC$, which concludes the proof of the first part.

For the second part, just observe that when $|\widetilde \cC|\ge 3$, the vertices in $\widetilde \cC$ divide its supporting cycle into paths without common colors, so there are at most $k$ such paths.
\end{proof}

For a positive integer $\ell$, we say that a cycle in $G$ is  \emph{separated into $\ell$ parts} if it can be divided into $\ell$ consecutive paths that use disjoint sets of colors. 
We say that it is separated into \emph{exactly} $\ell$ parts if it is separated into $\ell$ parts but not into $\ell+1$ parts. The following fact follows directly from the previous definition.

\begin{lemma}\label{lemma:CA.cycle}
Every cycle in $G$ supports at most one CA-component of size more than $1$. Moreover, a CA-component supported by a cycle has size $\ell$ if and only if its supporting cycle is separated into exactly $\ell$ parts.
\end{lemma}
\begin{proof}
Suppose that $\widetilde \cC_1$ and $\widetilde \cC_2$ are two distinct CA-components of sizes $\ell_1,\ell_2\ge 2$, respectively, which are supported by the same cycle, say $C$. Then, $\widetilde \cC_1$ and $\widetilde \cC_2$ must be disjoint. Moreover, the vertices of $\cC_1$ separate $C$ into $\ell_1$ parts, and the ones of $\cC_2$ separate $C$ into $\ell_2$ parts. It follows that the vertices of $\cC_1\cup \cC_2$ separate $C$ into $\ell_1+\ell_2$ parts, and so $\cC_1\cup \cC_2$ is a CA-component itself, a contradiction.

Thus, one cycle can support at most one CA-component. At the same time, if it supports a CA-component of size $\ell \ge 2$, it cannot be divided into $\ell+1$ parts as otherwise it 
would also support a CA-component of size more than $\ell$, which finishes the proof.
\end{proof}

The last important piece towards the proof of Theorem~\ref{thm:k>2}~\eqref{pt iii} is the following lemma.

\begin{lemma}\label{lem:separated cycles}
For every $m\in \{2,\dots,k\}$, denote by $Y_m$ the number of cycles in $G$ that are separated into exactly $m$ parts. Then, there are positive constants $\widetilde \beta_2$ and $\beta_3, \ldots, \beta_k$, such that 
\[(Y_2, \ldots, Y_k)\xrightarrow[n\to \infty]{d} \Po(\widetilde \beta_2)\otimes \bigotimes_{m=3}^k \Po(\beta_m).\]
\end{lemma}
\begin{proof}[Proof of Lemma~\ref{lem:separated cycles}]
For every $m\in \{2,\ldots, M\}$, denote by $Y_{m,M}$ the number of cycles in $G$ that are separated into exactly $m$ parts and having length at most $M$.

The first step of the proof is to show that for every $m\in \{2,\dots,k\}$, uniformly in $n$, 
\begin{equation}\label{conv.Ym.L1}
    \mathbb E[Y_m-Y_{m,M}] \xrightarrow[M\to \infty]{}  0.
    \end{equation}
To show this, note that a cycle of $G$ is separated into (at least) two parts if and only if there exists a nonempty subset $I\subseteq [k]$ different from $[k]$ such that one part of the cycle is contained in $G_I$ while the other part of the cycle is contained in $G^I$. For every $\ell \ge 3$, set $C_{2,\ell}$ to be the number of cycles of length $\ell$ in $G$ which are separated into (at least) two parts.

Then, using that to form a cycle of length $\ell$, one may choose its vertices in $\binom{n}{\ell}$ ways and order them in $\frac{(\ell-1)!}{2}$ ways, we have
\begin{equation}\label{eq:lem cycles}
\mathbb E [C_{2,\ell}] \le \binom{n}{\ell} \frac{(\ell-1)!}{2} \sum_{I\subseteq [k]} \sum_{m=1}^{\ell-1} \left(\frac{\sum_{i\in I} \lambda_i}{n}\right)^m \left(\frac{\sum_{i\in [k]\setminus I} \lambda_i}{n}\right)^{\ell-m}\le 2^k\cdot  (\lambda_k^*)^\ell.
\end{equation}
It follows that
\begin{equation*}
\mathbb E\left[Y_m - Y_{m,M}\right] \le \sum_{\ell= M+1}^n \mathbb E [C_{2,\ell}]\le 2^k  \sum_{\ell=M+1}^\infty (\lambda_k^*)^{\ell},
\end{equation*}
which goes to $0$ as $M\to \infty$ uniformly in $n$ since $\lambda_k^*<1$ by hypothesis, thus proving~\eqref{conv.Ym.L1}. 

The second step of the proof is to show that for every fixed $M\ge k$, one has 
\begin{equation}\label{conv.law.cycles}
(Y_{2,M},  \ldots, Y_{k,M})\xrightarrow[n\to \infty]{d} \bigotimes_{m=2}^k \Po(\beta_{m,M})
\end{equation}
for some positive constants $(\beta_{m,M})_{m=2}^k$. 
For $m \le k$, denote by $p_{m,\ell}$ the probability that a cycle 
of length $\ell$ in $G$ is separated into exactly $m$ parts. Let also $\widetilde Y_{m,\ell}$ denote the number of cycles of length $\ell$ which are separated into exactly $m$ parts. In particular,
$$Y_{m,M} = \sum_{\ell = \max(m,3)}^M \widetilde Y_{m,\ell}. $$ 
Adopting the notation of Lemma~\ref{lem:Bol}, observe also that for every $m\le k$, conditionally on $G$, 
$$\widetilde Y_{m,\ell}\ \mathop{=}^d\  \text{Bin}(C_\ell,p_{m,\ell}).$$  
Moreover, recall that $G$ is distributed as an Erd\H{o}s-R\'enyi random graph with parameters $n$ and
$p=(1+o(1))\cdot\tfrac{\Lambda}{n}$,  
so Lemma~\ref{lem:Bol} shows that 
$$(C_3,\dots,C_M) \xrightarrow[n\to \infty]{d} \bigotimes_{\ell = 3}^M Z_\ell$$ 
where for every $\ell\in \{3,\dots,M\}$, $Z_\ell$ is a Poisson random variable with parameter $\gamma_\ell=\tfrac{\Lambda^\ell}{2\ell}$. It follows that for every fixed $\ell$, (still writing with a slight abuse of notation $p_{m,\ell}$ for the limiting value of this probability as $n\to \infty$, see  Remark~\ref{rem:betas} for an explicit expression when $m=k$), 
$$(\widetilde Y_{2,\ell}, \dots, \widetilde Y_{k,\ell})\xrightarrow[n\to \infty]{d} 
\Big(\text{Bin}(Z_\ell,p_{2,\ell}), \dots,\text{Bin}(Z_\ell,p_{k,\ell})\Big), $$
which by the thinning property of the Poisson distribution (see e.g. Section 5.3 in~\cite{LP17}) is a vector of independent Poisson variables with parameters $(p_{m,\ell}\cdot \gamma_\ell)_{m=2}^\ell$. Summing over $\ell$ and using the independence of the variables $(Z_\ell)_{\ell= 3}^M$, we deduce that \eqref{conv.law.cycles} holds with 
$$\beta_{m,M} = \sum_{\ell = \max(m,3)}^M p_{m,\ell}\cdot \gamma_\ell.$$
The final step is to show that for every $m\in \{2,\dots,k\}$, the sequence $(\beta_{m,M})_{M\ge 3}$ is a bounded non-decreasing sequence, which therefore converges as $M\to \infty$ to some positive and finite constant. The fact that it is non-decreasing is straightforward by definition. On the other hand, by~\eqref{eq:lem cycles} we deduce that for every $m\in \{2,\dots,k\}$ and $M\ge 1$, 
$$\beta_{m,M} \le \liminf_{n\to \infty}\  \mathbb E[Y_{m,M}]\le \liminf_{n\to \infty} \ \mathbb E\left[\sum_{\ell=3}^{M} C_{2,\ell}\right]\le \frac{2^k}{1-\lambda_k^*},$$
showing that the sequence $(\beta_{m,M})_{M\ge 3}$ is bounded, which completes the proof.
\end{proof}

To finish the proof of Theorem~\ref{thm:k>2}~\eqref{pt iii}, note that by Lemma~\ref{cor:support.cycle} a.a.s.\ for every $m\ge 3$ we have $N_m = Y_m$, while $N_2$ is the sum of $Y_2$ and the repeated edges in $G$. Thus, using the notation of Lemma~\ref{lem:separated cycles} and Corollary~\ref{cor:Bol}, Theorem~\ref{thm:k>2}~\eqref{pt iii} follows with the constants $\beta_2 = \widetilde \beta_2 + \gamma_2$ and $(\beta_m)_{m=3}^k$.\hfill $\square$

\begin{remark}\label{rem:betas}
We note that while it is possible to provide explicit expressions for $\beta_2,\dots,\beta_k$ in terms of $\lambda_1,\dots,\lambda_k$, they tend to be more and more complicated as $\ell$ decreases from $k$ to $2$. However, one can provide a simple formula for $\beta_k$. We do this in the case $k\ge 3$; in fact, with the notation of Corollary~\ref{cor:Bol}, in the case $k=2$ one simply needs to add $\gamma_2$ to the final result to account for the number of repeated edges. 

By the previous proof, one has 
$$\beta_k = \sum_{\ell \ge k} p_{k,\ell} \cdot \gamma_\ell,$$ 
where $p_{k,\ell}$ is the limit (as $n\to \infty$) of the probability that a cycle of length $\ell$ in $G$ is separated into exactly $k$ parts, and $\gamma_\ell = \tfrac{\Lambda^\ell}{2\ell}$. To compute $p_{k,\ell}$, one needs to decide the lengths $s_1,\dots,s_k\ge 1$ of the portions of the cycle in colors $1,\dots,k$, respectively (with the constraint that $s_1+\dots+s_k=\ell$). Then, one needs to choose the starting vertex of the path colored in color $1$ (say when turning clockwise), for which there are $\ell$ choices, and the order of appearance of the other colors, for which there are $(k-1)!$ choices. Finally, note that as $n\to \infty$, the probability that an edge is colored in color $i$ tends to $\lambda_i/\Lambda$, which in total yields the formula 
$$p_{k,\ell} = \ell  (k-1)!\cdot \sum_{s_1+\dots+s_k=\ell}  \prod_{i=1}^k \left(\frac{\lambda_i}{\Lambda}\right)^{s_i}.$$ 
Altogether this gives 
$$\beta_k = \sum_{\ell \ge k} p_{k,\ell}\cdot \gamma_\ell = \frac{(k-1)!}{2}\sum_{\ell \ge k} \sum_{s_1+\dots+s_k=\ell}
\prod_{i=1}^k \lambda_i^{s_i} = \frac{(k-1)!}{2}\prod_{i=1}^k \Bigg(\sum_{j=1}^{\infty}
\lambda_i^j\Bigg) = \frac{(k-1)!}{2}\prod_{i=1}^k \frac{\lambda_i}{1-\lambda_i},$$ 
remembering for the third equality that the sum runs over indices $\{s_i\}_{i\in [k]}$ larger than or equal to $1$. 

For completeness, let us mention another slightly different way to compute $\beta_k$. Note first that since CA-components are supported by cycles or single edges, the expected number of CA-components of size $k$, or equivalently of cycles which are separated in exactly $k$ parts, is equal to 
$$\frac {1+o(1)}2 \sum_{i_1,\dots,i_k\in [k]} \sum_{u_1, \dots , u_k\in [n]}
\mathbb P(u_1 \xleftrightarrow{G_{i_1}} u_2, \dots, u_k\xleftrightarrow{G_{i_k}} u_1),$$ 
where the two sums run over $k$-tuples of ordered pairwise distinct elements of
$[k]$ and $[n]$, respectively, with $i_1=1$ (the factor $1/2$ coming from the fact that there are two possible ways to orient a cycle). Now, recall that for any pair of distinct vertices $u,v \in [n]$ and any $i\in [k]$, by~\eqref{prob.cond.connect} one has that
$$\mathbb P(v\in \cC_i(u)) = \frac{\mathbb E[|\cC_i(u)|-1]}{n-1}.$$
Thus, by induction we get that for any $(i_1,\dots,i_k)$,
\begin{align*}
\sum_{u_1,\dots,u_k} \mathbb P(u_1 \xleftrightarrow{G_{i_1}} u_2, \dots, u_k\xleftrightarrow{G_{i_k}} u_1)& = \frac{\mathbb E[|\cC_{i_k}(1)|-1]}{n-1}\cdot \sum_{u_1,\dots,u_k}
\mathbb P(u_1 \xleftrightarrow{G_{i_1}} u_2, \dots, u_{k-1}\xleftrightarrow{G_{i_{k-1}}} u_k) \\
& = n(n-1)\dots(n-k+1)\prod_{i=1}^k \frac{\mathbb E[|\cC_i(1)|-1]}{n-1}\\
& = (1+o(1)) \prod_{i=1}^k \mathbb E[|\cC_i(1)|-1], 
\end{align*}
where for the second equality we use that the number of choices for the sequence $(u_1,\dots,u_k)$ is $n(n-1)\dots(n-k+1)$. The formula follows since there are $(k-1)!$ ways to choose $i_2,\dots,i_k$, and with the notation of Section~\ref{sec:proof.i} 
$$\mathbb E[|\cC_i(1)|-1] = (1+o(1))\cdot \mathbb E[|\mathbf{GW}(\lambda_i)|-1] = (1+o(1))\cdot  \frac{\lambda_i}{1-\lambda_i},$$
where the last equality is derived from the fact that for every $d\ge 1$, the expected number of vertices at distance exactly $d$ from the root in $\mathbf{GW}(\lambda_i)$ is $\lambda_i^d$.
\end{remark}

\section{Conclusion}\label{sec:discussion}

In this paper, we characterized precisely the size of the largest CA-component in randomly colored Erd\H{o}s-R\'enyi random graphs in the entire supercritical and subcritical regimes, and in part of the intermediate regime as well. The most obvious open question that we leave concerns the size of the largest CA-component when $\lambda_{k-1}^* \ge 1 > \lambda_1^*$. The additional difficulty this point presents compared to the second part of Theorem~\ref{thm:k>2} is that in general, one cannot obtain the largest CA-component as an intersection of two independent random graphs. Nevertheless, we conjecture that an analogue of Theorem~\ref{thm:k>2}~\eqref{pt ii} holds in this case as well. Unfortunately, confirming this fact seems to be out of reach with our present techniques even in the simplest case when $k=3$. We remark that when $\lambda_m^* < 1$, by a statement similar to Lemma~\ref{lem:intersection} (that is also proved in a similar way) one may reduce the problem to the case of $k-m+1$ colors where $G^2,\ldots,G^{k-m+1}$ are all supercritical graphs while $G^1$ is subcritical. 

In the critical case when $\lambda_1^*=1<\lambda_2^*$, we suspect that the size of the largest CA-component divided by $n^{2/3}$ converges in distribution towards a non-degenerate random variable. The reason is that 
in $G^1$, the size of the largest component divided by $n^{2/3}$ converges in distribution, and the largest components in $G^2,\ldots,G^k$ are of linear order. However, the lack of independence makes it difficult to turn this heuristic into a rigorous proof.

Another possible direction could be to 
explore the case of other classical random graphs, or the closely related model of randomly vertex-color-avoiding random graph, which was initially considered in the literature~\cite{KDZ16,KDZ17} (and in which, as its name suggests, we color the vertices of the graph instead of the edges). In particular, as suggested in~\cite{KDZ16}, it could be interesting to study the effect of clustering (typical for random graphs with power law degree distributions, for example), as it arises in numerous recently introduced real-world networks models. Indeed, in this case, removing vertices with large degree could have a dramatic effect on the connectivity properties of the graph. 

\vspace{0.3cm}
\noindent \textbf{Acknowledgments.} We thank Dieter Mitsche for enlightening discussions, Bal\'azs R\'ath for a number of comments and corrections on a first version of this paper, and an anonymous referee for several useful remarks.

\bibliographystyle{plain}
\bibliography{Refs}

\end{document}